\documentclass[a4paper,leqno]{article}
\usepackage[centertags]{amsmath}
\usepackage{amsfonts}
\usepackage{amssymb}
\usepackage{color}
\usepackage{amsthm}
\usepackage[active]{srcltx} 

\newtheorem{theorem}{Theorem}
\newtheorem{Main-Th}{Main Theorem}
\newtheorem{corollary}{Corollary}
\newtheorem{proposition}{Proposition}
\newtheorem{lemma}{Lemma}
\theoremstyle{definition}
\newtheorem{definition}{Definition}

\theoremstyle{remark}
\newtheorem{remark}{Remark}

\newenvironment{Proof}{{\bf Proof:}}
{%
\mbox{}%
\nolinebreak%
\hfill%
\rule{2mm}{2mm}%
\medbreak%
\par%
}
\newcommand{\cqfd}
{%
\mbox{}%
\nolinebreak%
\hfill%
\rule{2mm}{2mm}%
\medbreak%
\par%
}

\numberwithin{equation}{section}

\renewcommand{\thefootnote}{\fnsymbol{footnote}}
\newcommand\Item[1][]{%
  \ifx\relax#1\relax  \item \else \item[#1] \fi
  \abovedisplayskip=0pt\abovedisplayshortskip=0pt~\vspace*{-\baselineskip}}

\title{\bf\large On natural symmetries on slit tangent bundles of Finsler manifolds}

\author{Mohamed Tahar Kadaoui Abbassi, Abderrahim Mekrami }
\date{}
\begin{document}

 \maketitle


\begin{abstract}%
In this paper, we introduce a broad class of metrics on the slit tangent bundle of Finsler manifolds, termed \emph{$F$-natural metrics}. These metrics parallel the well-established $g$-natural metrics on the tangent bundles of Riemannian manifolds and are constructed using six real functions defined over the domain of positive real numbers. We provide an in-depth characterization of conformal, homothetic, and Killing vector fields derived from specific lifts of vector fields and tensor sections on the slit tangent bundle, which is equipped with a general pseudo-Riemannian $F$-natural metric.
\medskip


\renewcommand{\thefootnote}{\arabic{footnote}}


\renewcommand{\thefootnote}{\arabic{footnote}}
\end{abstract}


\section*{Introduction}


Finsler geometry, characterized by direction-dependent metrics, finds applications across diverse fields where anisotropic modeling is essential. In physics, it extends General Relativity by accommodating direction-dependent spacetime structures, supporting models that address cosmological anisotropies. It also provides frameworks in quantum mechanics and field theory for path integrals and Lagrangians with direction-dependent interactions, and models light propagation in optics within anisotropic media like fiber optics and metamaterials. In robotics and control theory, it optimizes path planning across variable terrains, aiding navigation in complex environments. In information theory, Finsler metrics capture anisotropic noise, improving data compression and transmission. This versatility positions Finsler geometry as a valuable tool for domains where directional variability is pivotal. 

In Finsler geometry, the tangent bundle enables direction-dependent metrics by defining the Finsler function on each point’s entire tangent space, allowing metric variation with both position and direction. This space is essential for describing geodesics, curvature, and dynamic behavior through the energy function, acting as a Lagrangian. It also supports fundamental connections (e.g., Chern or Berwald) that generalize parallel transport and curvature, making the tangent bundle central to Finsler geometry’s direction-sensitive structure. This is why the study of the geometry of the tangent bundle of a Finslerian manifold as a pseudo-Riemannian manifold could provide insight on the geometry of the base manifold itself and could approach some aspects of the Finslerian geometry by means of technics and objects of Riemannian geometry. 

Surprisingly, in literature, few works deals with the geometry of tangent bundles of Finslerian manifolds as a matter of study (cf. \cite{Wu}, \cite{Rae} and \cite{Rae-Lat}). In this paper, we try to partially fill this gap by investigating some problems related to the pseudo-Riemannian geometry of the tangent bundle of Finslerian manifolds, focusing on some kinds of infinitesimal symmetries.  

Generally speaking, given a smooth manifold $M$, a \emph{symmetry} of a tensor field $T$ on $M$ is defined as a one-parameter group of diffeomorphisms that leaves $T$ invariant. Equivalently, a vector field $X$  generates such a group of diffeomorphisms if it satisfies $\mathcal{L}_X T=0$, where $\mathcal{L}$ denotes the Lie derivative on $M$. Within the framework of pseudo-Riemannian geometry, the study of various symmetries significantly deepens our understanding of the geometric properties of these spaces. For an extensive treatment of symmetries, we refer to the monograph \cite{Hall}. 

In the context of a pseudo-Riemannian manifold $(M,g)$, the investigation of conformal vector fields is particularly insightful for understanding the geometric structure of $(M,g)$, with notable applications in physics. A \emph{conformal vector field} $X$ generates a one-parameter group of conformal diffeomorphisms and is characterized by the existence of a smooth function $f$ on $M$, known as the \emph{potential function}, such that $\mathcal{L}_Xg=2fg$. When $f$ is constant (or zero), $X$ is called a \emph{homothetic} (or \emph{Killing}) vector field.

In Riemannian geometry, there has been extensive research on the infinitesimal symmetries of tangent bundles, spurred by the variety of pseudo-Riemannian metrics defined on these bundles. The well-known Sasaki metric initially dominated this field, but its rigid association with the base metric led to the exploration of alternative metrics, such as the Cheeger-Gromoll metric, Oproiu metrics, and the broader class of g-natural metrics. Significant efforts have been made to classify conformal and Killing vector fields on tangent bundles, resulting in numerous classification theorems (see \cite{Cim-Tez}, \cite{Hed-Bid}, \cite{Pey-Hey}, \cite{Ta2} for Sasaki metric, \cite{ASar}, \cite{Gez-Bil} for the Cheeger-Gromoll metric, \cite{Ta1} for the complete lift metric and \cite{AAC}, \cite{SR} for $g$-natural metrics).

When the base manifold is Finslerian, the situation becomes substantially different. Any metric constructed from the base Finsler metric can only be defined on the slit tangent bundle, making the study of the Riemannian geometry of slit tangent bundles particularly compelling. Metrics such as the Sasaki and Cheeger-Gromoll metrics have been considered on the slit tangent bundles of Finsler manifolds, and several of their properties have been examined (cf. \cite{Wu}, \cite{Rae-Lat} and \cite{Rae}). 

In this paper, we introduce the concept of F-natural metrics, which are analogues of g-natural metrics on the slit tangent bundles of Finsler manifolds. These metrics are spherically symmetric and depend on six real-valued functions defined on the set of positive numbers. We provide a detailed characterization of non-degenerate and Riemannian F-natural metrics and compute their associated Levi-Civita connections. As an application, we demonstrate that the geodesic vector field on the slit tangent bundle, equipped with any Riemannian F-natural metric, is incompressible (Theorem \ref{g-f-incomp}). Moreover, we establish necessary and sufficient conditions for the fibers of the slit tangent bundle to be totally geodesic (Theorem \ref{t-g-fib}). Specifically, we prove that for Kaluza-Klein type $F$-natural metrics (where horizontal and vertical distributions are orthogonal), the fibers are totally geodesic if and only if the base manifold is a Landsberg manifold.

The second part of this paper is dedicated to studying conformal, homothetic, and Killing vector fields on the slit tangent bundles of Finsler manifolds equipped with pseudo-Riemannian $F$-natural metrics. We provide a comprehensive characterization of Killing horizontal lifts of vector fields, proving that no proper conformal or homothetic horizontal lift vector fields exist. Additionally, we analyze the vertical lift vector fields for Kaluza-Klein type $F$-natural metrics and characterize conformal complete lift vector fields on slit tangent bundles of Finsler manifolds endowed with the Sasaki and Cheeger-Gromoll metrics. Notably, in the Sasaki metric case, all conformal complete lift vector fields are Killing, whereas no Killing complete lift vector fields exist in the Cheeger-Gromoll metric case. We also investigate vertical vector fields on the slit tangent bundle derived from the "skew symmetric transvections" of $(1,1)$-tensor fields on the base manifold, characterizing those that are conformal when the $F$-natural metric on the slit tangent bundle is the Sasaki or Cheeger-Gromoll metric. Specifically, we demonstrate that such homothetic vector fields must be Killing. As an application, we classify all the pseudo-Riemannian $F$-natural metrics on the slit tangent bundle of a Finsler manifold for which the Liouville vector field is conformal, homothetic, or Killing, and prove that the geodesic vector field cannot be conformal.

In our paper, we adopt the pullback formalism developed in the works of Grifone (\cite{Gri1}, \cite{Gri2}),  Lovas \cite{Lov}, Szilasi and Tóth \cite{Szi-Tot}, along with the general theory of connections on vector bundles, to revisit Finsler geometry, leading to an elegant presentation of the relevant formulas (with free-coordinate expressions). 


\section{Finsler geometry revisited}


Let $M$ be a connected $C^\infty$-manifold, and let $TM$ denote its tangent bundle.  The natural projection $\pi:TM \rightarrow M$ is given by associating each tangent vector with its base point. We define the \emph{slit tangent bundle} of $M$, denoted by $\widetilde{TM}$, by
$$\widetilde{TM}:=TM\setminus\{0\} = \{u \in T_xM | u \neq 0, x\in  M \}.$$
The slit tangent bundle $\widetilde{TM}$ is the tangent bundle with the zero section removed, and its natural projection $\pi_0:\widetilde{TM} \rightarrow M$ is simply the restriction of $\pi$ to $\widetilde{TM}$.


\subsection*{The pullback bundle $\pi_0^*TM$}


The \emph{pullback bundle} $\pi_0^*TM$ is the vector bundle obtained by pulling back $TM$ to $\widetilde{TM}$, with projection map $\pi_1:\pi_0^*TM \rightarrow \widetilde{TM}$ defined by $(x,u,v) \mapsto (x,u)$. The fiber of $\pi_0^*TM$  at any point $(x,u) \in \widetilde{TM}$ is a copy of $T_xM$, that is, $(\pi_0^*TM)_{(x,u)} \cong T_xM$. Hence, $\pi_0^*TM$ is called the \emph{pullback tangent bundle}.

For every local vector field $X$ on $M$, we can naturally induce a local section $\pi_0^*X$ of $\pi_0^*TM$, defined by: 
$$\pi_0^*X(x,u)=(x,u,X_x), \quad \textup{for all} \quad (x,u) \in \widetilde{TM}.$$
Moreover, the differential of the projection map $\pi_0$ induces a bundle map $\rho:\mathfrak{X}(\widetilde{TM}) \rightarrow \Gamma(\pi_0^*TM)$, given by 
$$\rho(X)(u)=(u,(d\pi_0)_u(X_u)), \quad \textup{for all} \quad u \in \widetilde{TM}.$$

Given a coordinate system $(U;x^i,i=1,...,n)$ on $M$ the induced coordinate system on $TM$ is $(\pi^{-1}(U);x^i,u^i,i=1,...,n)$,  and similarly on $\widetilde{TM}$, we have coordinates $(x^i,u^i,i=1,...,n)$) on $\pi_0^{-1}(U)$. The vector fields $\frac{\partial}{\partial x^i}$ on $U$ induce sections $\pi_0^*(\frac{\partial}{\partial x^i})$ on $\pi_0^{-1}(U)$, denoted by $\partial_i$. Specifically 
$$\partial_i(x,u)=(x,u,\frac{\partial}{\partial x^i}\big|_x), \quad \textup{for all}\quad (x,u) \in \pi_0^{-1}(U).$$ 
This gives us a local frame  $\{\partial_i, i=1,...,n\}$ for the pullback tangent bundle $\pi_0^*TM$.  

Similarly, we define the \emph{pull-back cotangent bundle} $\pi_0^*T^*M$, where the fiber at any point $(x,u)$ is a copy of the cotangent space $T^*_xM$. Hence, the pullback cotangent bundle is the dual of the pullback tangent bundle, with a local frame $\{\partial_i^*, i=1,...,n\}$, where  
$$\partial_i^*(x,u) =(x,u,dx^i(x)), \quad \textup{for all}\quad (x,u) \in \pi_0^{-1}(U).$$


\subsection*{Tensor sections on $\pi_0^*TM$}


For $k,l \in \mathbb{N}^2$, we define the tensor product bundle $(\pi_0^*TM)^{\otimes^k} \otimes (\pi_0^*T^*M)^{\otimes^l}$, whose sections are known as \emph{$(k, l)$-tensor sections} on  $\pi_0^*TM$. In particular, the \emph{canonical section} $\mathcal{U}$ of $\pi_0^*TM$ is defined by $\mathcal{U}(x,u):=(x,u,u)$, for all $(x,u) \in \widetilde{TM}$.  

A \emph{pseudo-Riemannian fiber-metric} (resp. \emph{Riemannian fiber metric}) on $\pi_0^*TM$ is a $(0,2)$-tensor section on $\pi_0^*TM$ whose restriction to each fiber of $\pi_0^*TM$ is non-degenerate (resp. positive definite). Locally, any $(k,l)$-tensor section $T$ can be expressed as 
$$T=\sum_{i_1,...,i_k=1}^n \sum_{j_1,...,j_l=1}^n T_{j_1,...,j_l}^{i_1,...,i_k}\partial_{i_1}\otimes ...\otimes \partial_{i_k}\otimes \partial_{j_1}^* \otimes ...\otimes \partial_{j_l}^*,$$
where $T_{j_1,...,j_l}^{i_1,...,i_k}$ are smooth functions on $\pi_0^{-1}(U)$. In particular, we have 
$$\mathcal{U}=\sum_{i=1}^n u^i\partial_i.$$


\subsection*{The vertical distribution- the vertical lift}


In the context of the geometry of tangent bundles, canonical vector fields can be derived from vector fields on the base manifold. Two such fields are the vertical lift and the complete lift. For a given vector field $X$ on $M$, the \emph{vertical lift} $X^v$ and the \emph{complete lift} $X^c$ are defined as follows:
$$X^v(df)=X(f) \quad \textup{and} \quad X^c(df)=d(X(f)), \quad \textup{for all} \quad f \in C^\infty(M).$$
In local coordinates, if $X=\sum_{i=1}^{n} X^i \frac{\partial}{\partial x^i}$, the vertical and complete lifts are given by:
$$X^v=\sum_{i=1}^{n} X^i \frac{\partial}{\partial u^i} \quad \textup{and} \quad X^c=\sum_{i=1}^{n} X^i \frac{\partial}{\partial x^i} +\sum_{i,j=1}^{n} \frac{\partial X^i}{\partial x^j}u^j \frac{\partial}{\partial u^i}.$$
The vertical lift of a vector is defined analogously. Given a vector $X \in T_xM$, its vertical lift to $TTM$ is the vector $X^v \in T_{(x,u)}TM$, defined by:
$$X^v=\frac{d}{dt}\Big|_{t=0}(u+tX),$$
where $u \in T_xM$ is the base point.

Moreover, for local coordinate vector fields $\frac{\partial}{\partial x^i}$, the vertical lift $\left(\frac{\partial}{\partial x^i}\right)^v$ is simply $\frac{\partial}{\partial u^i}$. Hence, the set $\{\frac{\partial}{\partial u^i};i=1,...,n\}$ forms a local frame for the vertical distribution of $TTM$.

Since both $T\widetilde{TM}$ and $T\pi_0^*TM$ are vector bundles, we can define their respective \emph{vertical subbundles}. For each $(x,u) \in \widetilde{TM}$, the \emph{vertical subspace} $\mathcal{V}_{(x,u)}\widetilde{TM}$ is defined as the kernel of the projection $(d\pi_0)_{(x,u)}:T_{(x,u)}\widetilde{TM} \rightarrow T_xM$. Similarly, for $(x,u,v) \in \pi_0^*TM$, the \emph{vertical subspace}  $\mathcal{V}_{(x,u,v)}\pi_0^*TM$ is the kernel of the projection $(d\pi_1)_{(x,u,v)}:T_{(x,u,v)}\pi_0^*TM \rightarrow T_{(x,u)}\widetilde{TM}$. 

The \emph{vertical subbundles} $\mathcal{V}\widetilde{TM}$ and $\mathcal{V}\pi_0^*TM$ are given by
$$\mathcal{V}\widetilde{TM} =\dot{\bigcup}_{(x,u) \in \widetilde{TM}}\mathcal{V}_{(x,u)}\widetilde{TM}, \qquad \mathcal{V}\pi_0^*TM =\dot{\bigcup}_{(x,u,v) \in \pi_0^*TM}\mathcal{V}_{(x,u,v)}\pi_0^*TM.$$ 
The elements of these subbundles are referred to as \emph{vertical vectors}. 

The \emph{vertical lift} of a vector $X \in T_xM$ (resp. a section $W \in (\pi_0^*TM)_{(x,u)}$) at $u \in T_xM \setminus\{0_x\}$ (resp. $v \in (\pi_0^*TM)_{(x,u)}$) is the unique vertical vector $X^v$ in $T_{(x,u)}\widetilde{TM}$ (resp. $W^v$ in $T_{(x,u,v)}\pi_0^*TM$) defined by 
$$X^v=\frac{d}{dt}\Big|_0(u +tX) \quad (\textup{resp. } W^v=\frac{d}{dt}\Big|_0(v +tW)).$$
In a similar fashion, \emph{the vertical lift} of a vector field $X$ on $M$ (resp. section $\sigma$ of $\pi_0^*TM$) is the vector field $X^v$ on $\widetilde{TM}$ (resp. $\sigma^v$ on $\pi_0^*TM$) whose value at any point $(x,u) \in \widetilde{TM}$ (resp. $(x,u,v) \in \pi_0^*TM$) is the vertical lift of the vector $X_x$ (resp. $\sigma(x,u)$) to $T\widetilde{TM}$ at $u$ (resp. to $T\pi_0^*TM$ at $v$). 

It is easy to see that the vertical lift of the vector (resp. local vector field) $\frac{\partial}{\partial x^i}\big|_x$ at $(x,u)$ (resp. $\frac{\partial}{\partial x^i}$) is $\frac{\partial}{\partial u^i}\big|_{(x,u)}$ (resp. $\frac{\partial}{\partial u^i}$), so that $\{\frac{\partial}{\partial u^i}\big|_{(x,u)}, i=1,...,n\}$ (resp. $\{\frac{\partial}{\partial u^i}, i=1,...,n\}$) is a basis of $\mathcal{V}_u\widetilde{TM}$ (resp. is a local frame of $\mathcal{V}\widetilde{TM}$). Similarly, the vertical lift of the vector (resp. local section) $\partial_i(x,u)$ at $(x,u,v)$ (resp. $\partial_i$) is $\frac{\partial}{\partial v^i}\big|_{(x,u,v)}$ (resp. $\frac{\partial}{\partial v^i}$), so that $\{\frac{\partial}{\partial v^i}\big|_{(x,u,v)}, i=1,...,n\}$ (resp. $\{\frac{\partial}{\partial v^i}, i=1,...,n\}$) is a basis of $\mathcal{V}_{(x,u,v)}\pi_0^*TM$ (resp. is a local frame of $\mathcal{V}\pi_0^*TM$).


\subsection*{Finsler metrics}


A \emph{Finsler manifold} is a pair $(M,F)$, where $M$ is a manifold and $F : TM \rightarrow [0,+\infty)$ is a function called the \emph{Finsler norm}, that satisfies the following properties:
\begin{enumerate}
  \item \emph{Smoothness:} $F$ is smooth on $\widetilde{TM}$;
  \item \emph{Homogeneity:} $F$ is positively 1-homogeneous on the fibers of tangent bundle $TM$, i.e. $F(\lambda u)=\lambda F(u)$ for all $\lambda>0$ and $u \in TM$;
  \item \emph{Minkowski norm:} For every $x \in M$, the restriction $F|_{T_xM} : T_xM \rightarrow [0,+\infty)$ is a \emph{Minkowski norm} on $T_xM$. That is, for each nonzero $u \in T_xM$,  the symmetric bilinear form $g_u:T_xM \times T_xM \rightarrow \mathbb{R}$, defined by 
      $$g_u(v,w):=\frac12 \frac{\partial^2F^2(u+tv+sw)}{\partial t \partial s}\big|_{(t,s)=(0,0)}$$
      is positive definite. In other words $g$ is a Riemannian fiber-metric on pullback tangent bundle $\pi_0^*TM$.
\end{enumerate}
This bilinear form $g_u$ induced by the Finsler norm is called the \emph{fundamental tensor} of the Finsler structure, and it varies smoothly over the slit tangent bundle $\widetilde{TM}$. The finsler manifold is sometimes denoted by $(M,F,g)$.

A Finsler metric also induces a $(0,3)$-tensor section $C$ on $\pi_0^*TM$, known as the \emph{Cartan tensor}, which is defined as
$$C_u(v,w,z):=\frac14 \frac{\partial^3F^2(u+tv+sw+rz)}{\partial t \partial s \partial r}\big|_{(t,s,r)=(0,0,0)},$$ 
for all $u \in T_xM \setminus \{0\}$ and $v,w,z \in T_xM$.

Locally, the fundamental tensor $g$ and the Cartan tensor $C$ can be expressed as: 
$$g=\sum_{i,j=1}^ng_{ij} \partial_i^* \otimes \partial_j^* \quad \textup{and} \quad C=\sum_{i,j,k=1}^nC_{ijk} \partial_i^* \otimes \partial_j^* \otimes \partial_k^*,$$
where 
$$g_{ij}=\frac12 \frac{\partial^2 F^2}{\partial u^i \partial u^i} \quad \textup{and} \quad C_{ijk}=\frac14 \frac{\partial^3 F^2}{\partial u^i \partial u^i \partial u^k}.$$
By contracting the Cartan tensor with respect to the fundamental tensor $g$, we obtain another $(2,1)$-tensor section $\bar{C}$, also referred to as the \emph{Cartan tensor}, whose components are given by: 
$$\bar{C}_{ij}^k=g^{kl}C_{lij},$$
where $(g^{kl})$ denotes the inverse matrix of $(g_{kl})$.


\subsection*{The Chern connection}


There is a unique connection $\nabla$ on the pullback tangent bundle $\pi_0^*TM$, which satisfies the following properties: 
\begin{enumerate}
  \item \emph{almost compatibility with $g$}: For all vector fields $X \in \mathfrak{X}(\widetilde{TM})$ and sections $s_1, s_2 \in \Gamma(\pi_0^*TM)$, we have
  $$X(g(s_1,s_2))=g(\nabla_{X}s_1,s_2) +g(s_1,\nabla_{X}S_2) +2C(s_1,s_2,\nabla_X \mathcal{U}),$$
  where $\mathcal{U}$ is the canonical section of $\pi_0^*TM$;
  \item \emph{Torsion-free}:  The \emph{torsion} of the connection, defined by:
  $$T(X,Y)= \nabla_X \rho(Y) -\nabla_Y \rho(X) -\rho([X,Y]),$$ 
  for all $X,Y \in \mathfrak{X}(\widetilde{TM})$, is identically zero.
\end{enumerate}

Let $\Gamma_{Ij}^k$, $I=1,...,2n$, $j,k=1,...,n$, be the Christoffel symbols of the Chern connection $\nabla$ defined by the following identities
$$\nabla_{\frac{\partial}{\partial x^i}}\partial_j =\sum_{k=1}^n\Gamma_{ij}^k \partial_k, \quad \nabla_{\frac{\partial}{\partial u^i}}\partial_j =\sum_{k=1}^n\Gamma_{n+ij}^k \partial_k.$$
Due to the torsion-free property of the Chern connection, we have:
$$\Gamma_{n+ij}^k =0 \quad \textup{and} \quad \Gamma_{ij}^k=\Gamma_{ji}^k , \quad \textup{for all} \quad i,j,k=1,...,n.$$


\subsection*{The horizontal distribution- the horizontal lift}


With respect to the Chern connection $\nabla$ on the pullback tangent bundle $\pi_0^*TM$, we can define the horizontal distribution $\mathcal{H}\pi_0^*TM= \dot{\bigcup}_{(x,u,v) \in \pi_0^*TM} \mathcal{H}_{(x,u,v)}\pi_0^*TM$, which allows the following decomposition of the pullback tangent bundle:
$$T_{(x,u,v)}\pi_0^*TM=\mathcal{V}_{(x,u,v)}\pi_0^*TM \oplus \mathcal{H}_{(x,u,v)}\pi_0^*TM.$$
There is a natural isomorphism between the horizontal subspace $\mathcal{H}_{(x,u,v)}\pi_0^*TM$ and $T_{(x,u)}\widetilde{TM}$ given by the map
$$(d\pi_1)_{(x,u,v)}:\mathcal{H}_{(x,u,v)}\pi_0^*TM \rightarrow T_{(x,u)}\widetilde{TM},$$
where $\pi_1:\pi_0^*TM \rightarrow \widetilde{TM}$, $(x,u,v) \mapsto (x,u)$, is the first projection. This isomorphism enables the definition of the \emph{horizontal lift} $X^H$ at $(x,u,v)$ of any vector $X \in T_{(x,u)}\widetilde{TM}$ by 
$$X^H:=(d\pi_1)_{(x,u,v)}^{-1}(X).$$ 
In local coordinates, if $X=\sum_{i=1}^n X^i\frac{\partial}{\partial x^i}\big|_{(x,u)} + \sum_{i=1}^n X^{n+i}\frac{\partial}{\partial u^i}\big|_{(x,u)}$, the horizontal lift is expressed as:   
$$X^H=\sum_{i=1}^n X^i\frac{\partial}{\partial x^i}\Big|_{(x,u,v)} + \sum_{i=1}^n X^{n+i}\frac{\partial}{\partial u^i}\Big|_{(x,u,v)} -\sum_{i,j,k=1}^n \Gamma_{jk}^i X^ju^k \frac{\partial}{\partial v^i}\Big|_{(x,u,v)}.$$
Horizontal lifts of (local) vector fields on $\widetilde{TM}$ are defined similarly. Note that 
$$\big(\frac{\partial}{\partial u^i}\big)^H = \big(\frac{\partial}{\partial x^i}\big)^v=\frac{\partial}{\partial u^i}$$ 
and, for local vector fields $\frac{\partial}{\partial x^i}$, the horizontal lift can be denoted as: 
$$\frac{\delta}{\delta x^i}:= \big(\frac{\partial}{\partial x^i}\big)^H =\frac{\partial}{\partial x^i} -\sum_{j,k=1}^n \Gamma_{ji}^k u^j \frac{\partial}{\partial v^k}.$$
Hence, the set $\{\frac{\partial}{\partial u^i},\frac{\delta}{\delta x^i}; i=1,...,n\}$ forms a local frame on $\pi_0^{-1}(U)$,  and we can decompose: 
$$T_{(x,u)}\widetilde{TM}= \mathcal{V}_{(x,u)}\widetilde{TM} \oplus \mathcal{H}_{(x,u)}\widetilde{TM},$$
where $\mathcal{H}_{(x,u)}\widetilde{TM}:= \textup{span} \{\frac{\delta}{\delta x^i}\big|_{(x,u)}; i=1,...,n\}$.

Let $X$ be a vector (resp. vector field) on $M$, expressed in a coordinates system $(U;x^1,...,x^n)$ as $X=\sum_{i=1}^nX^i \frac{\partial}{\partial x^i}\big|_x$ (resp. $X=\sum_{i=1}^nX^i \frac{\partial}{\partial x^i}$). The \emph{horizontal lift} of $X$ at $u=\sum_{i=1}^nu^i \frac{\partial}{\partial x^i}\big|_x \in T_xM \setminus\{0_x\}$ (resp. to $T\widetilde{TM}$) is the vector in $T_u\widetilde{TM}$ (resp. the vector field on $\widetilde{TM}$) given, in the coordinates system $(\pi_0^{-1}(U),x^1,...,x^n,u^1,...,u^n)$, by
$$X^h=\sum_{i=1}^nX^i \frac{\delta}{\delta x^i}\Big|_{(x,u)} \quad (\textup{resp. } X^h=\sum_{i=1}^n(X^i \circ \pi_0) \frac{\delta}{\delta x^i}).$$
Thus, locally, we have
$$X^h= \sum_{i=1}^nX^i \frac{\partial}{\partial x^i} -\sum_{i=1}^n \Gamma_{kj}^i X^ju^k\frac{\partial}{\partial u^i}.$$

From sections of $\pi_0^*TM$, we can define horizontal and vertical vector fields on $\widetilde{TM}$ as follows: let $\sigma \in \Gamma(\pi_0^*TM)$, then $h\sigma$ (resp. $v\sigma$) is the horizontal (resp. the vertical) vector fields on $\widetilde{TM}$ such that $h\sigma(u)$ (resp. $v\sigma(u)$) is the horizontal (resp. vertical) lift at $u$ of the vector $\pi_2(\sigma(u))$, for all $u \in \widetilde{TM}$. This gives rise to a map $h:\Gamma(\pi_0^*TM) \rightarrow \mathfrak{X}(\widetilde{TM})$ (resp. $v:\Gamma(\pi_0^*TM) \rightarrow \mathfrak{X}(\widetilde{TM})$). We can also define $h\{(x,u,v)\}$ (resp. $v\{(x,u,v)\}$), for $(x,u,v) \in \pi_0^*TM$, as the horizontal (resp. vertical) lift at $(x,u)$ of $(x,v) \in M_x$.

If we denote by $\tilde{\rho}$ the restriction of $\rho$ to the subbundle $\Gamma(\mathcal{H}\widetilde{TM})$ of horizontal vector fields on $\widetilde{TM}$, then it is easy to see that 
$$\tilde{\rho} \circ h=\rho \circ h=\textup{Id}_{\Gamma(\pi_0^*TM)} \quad \textup{and} \quad h \circ \tilde{\rho} =\textup{Id}_{\Gamma(\mathcal{H}\widetilde{TM})}.$$

If we denote by the same notation $X^c$ the restriction to $\widetilde{TM}$ of the complete lift of a vector field $X$ on $M$, then it is easy to see that
$$X^c=X^h +v\left\{\nabla_\zeta X\right\},$$
where $\zeta$ is the vector field on $\widetilde{TM}$ defined locally by $\zeta=\sum_{i=1}^{n}u^i\frac{\delta}{\delta x^i}$.


\subsection*{The curvature}


The curvature associated with the Chern connection $\nabla$ is defined by the map $R: \Gamma(T\widetilde{TM}) \times \Gamma(T\widetilde{TM}) \times \Gamma(\pi_0^*TM) \rightarrow \Gamma(\pi_0^*TM)$, given by
$$R(W,Z)s:= \nabla_W \nabla_Z s -\nabla_Z \nabla_W s -\nabla_{[W,Z]}s,$$
for all $Z, W \in \Gamma(T\widetilde{TM})$ and $s \in \Gamma(\pi_0^*TM)$. 

If either $Z$ or $W$ is vertical, then $R(W,Z)=0$. Thus, the curvature is entirely determined by its restriction $R: \Gamma(\mathcal{H}\widetilde{TM}) \times \Gamma(\mathcal{H}\widetilde{TM}) \times \Gamma(\pi_0^*TM) \rightarrow \Gamma(\pi_0^*TM)$ to the horizontal vector fields. In local coordinates, let $R_{ijk}^l$ represent the components of $R$ in the local frame $\{\frac{\delta}{\delta x^i};i=1,...,n\}$ and $\{\partial_i;i=1,...,n\}$ of $\Gamma(\mathcal{H}\widetilde{TM})$ and $\Gamma(\pi_0^*TM)$, respectively. We have 
$$R\left(\frac{\delta}{\delta x^i},\frac{\delta}{\delta x^j}\right)\partial_k=\sum_{l=1}^n R_{kij}^l \partial_l,$$ 
with
\begin{equation}\label{curv-loc}
  R_{kij}^l= \frac{\delta \Gamma^l_{jk}}{\delta x^i} -\frac{\delta \Gamma^l_{ik}}{\delta x^j} +\Gamma_{is}^l \Gamma_{jl}^s -\Gamma_{js}^l \Gamma_{il}^s.
\end{equation}

By contracting the curvature $R$ with respect to the Finsler metric $g$, we obtain the \emph{Riemannian curvature}, denoted also by $R$, which is defined by $R: \Gamma(\mathcal{H}\widetilde{TM}) \times \Gamma(\mathcal{H}\widetilde{TM}) \times \Gamma(\pi_0^*TM) \times \Gamma(\pi_0^*TM) \rightarrow C^\infty(\widetilde{TM})$, $R(W,Z,s_1,s_2)=g(R(W,Z)s_1,s_2)$.

The $(0,3)$-tensor sections $C$ and $L$ are symmetric in their three arguments and 
$$R(Z,W,s_1,s_2)=-R(W,Z,s_1,s_2), $$
for all $(Z,W,s_1,s_2) \in \Gamma(\mathcal{H}\widetilde{TM}) \times \Gamma(\mathcal{H}\widetilde{TM}) \times \Gamma(\pi_0^*TM) \times \Gamma(\pi_0^*TM)$.
We have also the following homogeneity properties:
$$g_{\lambda u}= g_u, \qquad g(\mathcal{U},\mathcal{U})=F^2=:r^2, \qquad C(\mathcal{U},.,.)=L(\mathcal{U},.,.)=0,$$ 
for all $\lambda>0$ and $u \in \widetilde{TM}$.

To make a link between the Finsler tools defined here and those defined in classical references of Finsler Geometry, we shall use the following identifications: 
\begin{itemize}
  \item Since, for all $(x,u) \in \widetilde{TM}$, the fiber $\mathcal{H}_{(x,u)}\widetilde{TM}$ of the vector bundle $\mathcal{H}\widetilde{TM}$ is isomorphic to the fiber $(\pi_0^*TM)_u$ of $\pi_0^*TM$, the linear isomorphism being given by identifying $\frac{\delta}{\delta x^i}\big|_u$ with $\partial_i|_u$, $i=1,...,n$, it follows that $\Gamma(\mathcal{H}\widetilde{TM})$ is identified to $\Gamma(\pi_0^*TM)$. With this identification, the curvature (resp. Riemannian curvature) $R$, associated to the Chern connection, can be identified with the $(1,3)$-tensor section (resp. $(0,4)$-tensor section) $\mathcal{R}$ given by
      $$\mathcal{R}(\sigma_1,\sigma_2)\sigma_3=R(h\sigma_1,h\sigma_2)\sigma_3$$ 
      $$(\textup{resp.}\quad \mathcal{R}(\sigma_1,\sigma_2,\sigma_3,\sigma_4)=R(h\sigma_1,h\sigma_2,\sigma_3,\sigma_4)),$$
      for all $\sigma_1,\sigma_2,\sigma_3,\sigma_4 \in \Gamma(\pi_0^*TM)$. Locally, $\mathcal{R}$ is expressed as
      $$\mathcal{R}(\partial_i,\partial_j)\partial_k=\sum_{l=1}^n R_{kij}^l \partial_l \quad (\textup{resp.}\quad  \mathcal{R}(\partial_i,\partial_j,\partial_k,\partial_l)= \sum_{\lambda=1}^n g_{l\lambda}R_{kij}^\lambda),$$
      where $R_{kij}^l$ is defined by \eqref{curv-loc}.
  \item Every vector field $X$ on $M$, given locally as $X|_U=\sum_{i=1}^nX^i \frac{\partial}{\partial x^i}$, induces the section $\pi_0^*X$ of $\pi_0^*TM$ given locally by $\pi_0^*X|_U=\sum_{i=1}^n X^i\partial_i$. So, $X$ can be viewed as a section of $\pi_0^*TM$ whose restriction to any $T_xM$, $x \in M$, is constant. 
      Hereafter, we shall write $X$ for $\pi_0^*X$, for any vector field on $M$. For example, we shall write $g(X,Y)$, $g(X,\mathcal{U})$ and $\mathcal{R}(X,Y)\mathcal{U}$ for $g(\pi_0^*X,\pi_0^*Y)$, $g(\pi_0^*X,\mathcal{U})$ and $\mathcal{R}(\pi_0^*X,\pi_0^*Y)\mathcal{U}$, respectively, when $X$ and $Y$ are vector fields on $M$.
\end{itemize}


\subsection*{Other useful tensor sections}


Several additional tensor sections arise from the Finsler metric $g$ and the curvature tensor. For example, the \emph{Landsberg tensor} $\bar{L}$ is a $(1,2)$-tensor section on $\pi_0^{-1}(U)$ that measures the lack of symmetry in the covariant derivative of the Finsler metric. It is defined locally by: 
$$L_{ij}^k:=\sum_{l=1}^n u^l \frac{\partial \Gamma_{lj}^k}{\partial u^i}, \quad i,j,k=1,...,n,$$
and expressed as: 
$$\bar{L}=\sum_{i,j,k=1}^n L_{ij}^k \partial_k \otimes \partial_i^* \otimes \partial_j^*.$$ 
The contraction with respect to the Finsler metric gives rise to a $(0,3)$-tensor section, denoted by $L$ and also called the \emph{Landsberg tensor}, defined by
$$L(X,Y,Z)=g(\bar{L}(X,Y),Z), \quad \textup{for all } X,Y,Z \in (\pi_0^*TM)_u \quad \textup{and } u\in\widetilde{TM}.$$
Locally, we have $L_{ijk}=\sum_{l=1}^n g_{kl}L_{ij}^l$.

Another useful tensor is the \emph{Berwald curvature}, a $(0,4)$-tensor section $B$ on $\pi_0^*TM$, defined by:
$$B(\sigma_1,\sigma_2,\sigma_3,\sigma_4):=-C(\sigma_1,\sigma_2,\mathcal{R}(\sigma_3,\sigma_4)\mathcal{U}),$$
for all $\sigma_1,\sigma_2,\sigma_3,\sigma_4 \in \Gamma(\pi_0^*TM)$. $B$ induces a $(1,2)$-tensor section $\bar{B}$ on $\pi_0^*TM$ by
$$g(\bar{B}(\sigma_1,\sigma_2),\sigma_3):=B(\sigma_1,\sigma_2,\sigma_3,\mathcal{U})=C(\sigma_1,\sigma_2,\mathcal{R}(\sigma_3,\mathcal{U})\mathcal{U}),$$
for all $\sigma_1,\sigma_2,\sigma_3 \in \Gamma(\pi_0^*TM)$.


\subsection*{Some useful formulas}


In this section, we present some useful lemmas that will aid in various calculations related to Finsler geometry.

\begin{lemma}\label{lem1}
  For any vector fields $X, Y \in \mathfrak{X}(M)$, we have the following identities
  \begin{enumerate}
    \item $\nabla_{X^v}Y =0$;
    \item $\nabla_{X^h}\mathcal{U}= 0$;
    \item $\nabla_{X^v}\mathcal{U}= X$;
    \item $\nabla_{X^c}Y =\nabla_{X^h}Y$, where $\zeta$ is the geodesic vector field on $\widetilde{TM}$;
    \item $\nabla_{X^c}\mathcal{U}= \nabla_\zeta X$;
    \item $[X,Y]= \nabla_{X^h}Y -\nabla_{Y^h}X$.
  \end{enumerate}
\end{lemma}

\begin{lemma}\label{lem2}
  For any vector field $X \in \mathfrak{X}(M)$, sections $s_1, s_2 \in \Gamma(\pi_0^*TM)$ and smooth function $f \in C^\infty(]0,+\infty[)$, we have
  \begin{enumerate}
    \item $X^h(f \circ r^2)=X^h(f \circ g(\mathcal{U},\mathcal{U}))=0$;
    \item $X^v(f \circ r^2)=X^v(f \circ g(\mathcal{U},\mathcal{U}))=2f^\prime \circ r^2 . g(X,\mathcal{U})$;
    \item $X^h(g(s_1,\mathcal{U}))=g(\nabla_{X^h}s_1,\mathcal{U})$;
    \item $X^v(g(s_1,\mathcal{U}))=g(s_1,X)$;
    \item $X^h(g(s_1,s_2))=g(\nabla_{X^h}s_1,s_2) +g(s_1,\nabla_{X^h}s_2)$;
    \item $X^v(g(s_1,s_2))= g(\nabla_{X^v}s_1,s_2) +g(s_1,\nabla_{X^v}s_2)+2C(X,s_1,s_2)$.\\
    In particular, $X^v(g(Y,Z))= 2C(X,Y,Z)$, for all $Y, Z \in \mathfrak{X}(M)$.
  \end{enumerate}
\end{lemma}

\begin{lemma}\label{lem3}
  For all $X, Y \in \mathfrak{X}(M)$, the following hold
  \begin{enumerate}
    \item $[X^h,X^h]= [X,Y]^h -v\{\mathcal{R}(X,Y)\mathcal{U}\}$;
    \item $[X^h,Y^v]= \nabla_{X^h} Y +v\{\bar{L}(X,Y)\}$;
    \item $[X^v,Y^v]=0$.
  \end{enumerate}
\end{lemma}

\begin{lemma}\label{lem4}
  For all sections $\sigma_1,\sigma_2, \sigma_3,\sigma_4 \in \Gamma(\pi_0^*TM)$, the following expressions for the Riemann curvature hold:
  \begin{enumerate}
    \item $g(\mathcal{R}(\sigma_1,\sigma_2)\mathcal{U},\mathcal{U})=0$;
    \item $g(\mathcal{R}(\mathcal{U},\sigma_1)\mathcal{U},\sigma_2)=g(\mathcal{R}(\mathcal{U},\sigma_2)\mathcal{U},\sigma_1)$;
    \item $g(\mathcal{R}(\sigma_1,\sigma_2)\sigma_3,\sigma_4)+g(\mathcal{R}(\sigma_1,\sigma_2)\sigma_4,\sigma_3) +2C(\sigma_2,\sigma_1,\mathcal{R}(\sigma_3,\sigma_4)\mathcal{U}) =0$;
    \item $\mathcal{R}(\sigma_1,\sigma_2)\sigma_3 +\mathcal{R}(\sigma_2,\sigma_3)\sigma_1 +\mathcal{R}(\sigma_3,\sigma_1)\sigma_2=0$;
    \Item \begin{flalign*}
            & g(\mathcal{R}(\sigma_1,\sigma_2)\sigma_3,\sigma_4)-g(\mathcal{R}(\sigma_3,\sigma_4)\sigma_1,\sigma_2)= \\
             = & C(\sigma_3,\sigma_4,\mathcal{R}(\sigma_1,\sigma_2)\mathcal{U})- C(\sigma_1,\sigma_2,\mathcal{R}(\sigma_3,\sigma_4)\mathcal{U})  +C(\sigma_3,\sigma_2,\mathcal{R}(\sigma_4,\sigma_1)\mathcal{U}) \\
              & +C(\sigma_4,\sigma_1,\mathcal{R}(\sigma_3,\sigma_2)\mathcal{U}) +C(\sigma_2,\sigma_4,\mathcal{R}(\sigma_1,\sigma_3)\mathcal{U}) +C(\sigma_1,\sigma_3,\mathcal{R}(\sigma_2,\sigma_4)\mathcal{U}).
          \end{flalign*}
  \end{enumerate}
\end{lemma}


\subsection*{The covariant derivative of a $(1,1)$-tensor section}


The concept of covariant derivatives can be extended to tensor sections of any type on $\pi_0^*TM$. For a $(1,1)$-tensor section $P$ on $\pi_0^*TM$ and a vector field $W \in \mathfrak{X}(\widetilde{TM})$, the \emph{covariant derivative} $\nabla_W P$ of $P$ is the $(1,1)$-tensor section on $\pi_0^*TM$ defined by
$$(\nabla_W P)(\sigma):= \nabla_W (P(\sigma)) -P(\nabla_W\sigma),$$
for any section $\sigma \in \Gamma(\pi_0^*TM)$.


\subsection*{The second covariant derivative}


Finally, the notion of covariant derivatives can be extended to second covariant derivatives with respect to the Chern connection. Let $\sigma \in \Gamma(\pi_0^*TM)$. The \emph{second covariant derivative} of $\sigma$, denoted $\nabla^2\sigma$, is the $C^\infty(\widetilde{TM})$-bilinear mapping $\mathfrak{X}(\widetilde{TM}) \times \mathfrak{X}(\widetilde{TM}) \rightarrow \Gamma(\pi_0^*TM)$ defined by
$$\nabla^2 \sigma(V,W)= \nabla_{W}\nabla_{V} \sigma -\nabla_{h\{\nabla_W \rho(V)\}} \sigma,$$
for all $V, W \in \mathfrak{X}(\widetilde{TM})$. 

Note that by the $C^\infty(\widetilde{TM})$-bilinearity, we can define  $\nabla^2 \sigma$ pointwise, i.e. if $(x,u) \in \widetilde{TM}$, then we can define $(\nabla^2 \sigma)_{(x,u)}:(\widetilde{TM})_{(x,u)} \times (\widetilde{TM})_{(x,u)} \rightarrow \pi_0^*TM$. 

It is straightforward to verify that the second covariant derivative satisfies the following identity:
$$\nabla^2\sigma(V,W) -\nabla^2\sigma(W,V)=R(W,V)\sigma,$$
for any $\sigma \in \Gamma(\pi_0^*TM)$ and $V, W \in \mathfrak{X}(\widetilde{TM})$.


\subsection*{Conformal vector fields}


A vector field $\xi$ on a Finsler manifold $(M,F,g)$ is called a \emph{conformal vector field} if there exists a smooth function $f$ on $\widetilde{TM}$ such that 
$$\tilde{\mathcal{L}}_{X^c}g=2fg,$$ 
where $\tilde{\mathcal{L}}$ denotes the Lie derivative (cf. e.g. \cite{Szi-Tot} for the definition and properties of Lie derivative in the context of Finsler geometry). If $f$ is constant, the vector field $\xi$ is said to be \emph{homothetic}, and if $f=0$, it is called a \emph{Killing vector field}. 

The Lie derivative is characterized by the identity:
$$\tilde{\mathcal{L}}_{\xi^c}g(X,Y)=g(\nabla _{X^{h}}\xi ,Y)+g(\nabla _{Y^{h}}\xi ,X)+2C(X,Y,\nabla_{\zeta}\xi),$$
for all $X,Y \in \mathfrak{X}(M)$.


\section{$F$-natural metrics on slit tangent bundles of Finsler manifolds}


In this section, we define a family of metrics on the slit tangent bundle $\widetilde{TM}$ of a Finsler manifold $(M,F,g)$, which we refer to as $F$-natural metrics. These metrics are analogous to the $g$-natural metrics on the tangent bundles of Riemannian manifolds (cf. \cite{AS}, \cite{Abb-Sar4}).

\begin{definition}
  A metric $G$ on $\widetilde{TM}$ is called an $F$-natural metric if there are six functions $\alpha_i,\beta_i: ]0,+\infty[ \rightarrow \mathbb{R}$, $i=1,2,3$, such that, for all $X, Y \in \mathfrak{X}(M)$, we have
  \begin{equation}\label{F-nat-metr}
    \left\{\begin{array}{l}
             G(X^h,Y^h)= (\alpha_1 +\alpha_3)\circ r^2 g(X, Y) +(\beta_1+\beta_3) \circ r^2 g(X, \mathcal{U})g(Y, \mathcal{U}),\\
             G(X^h,Y^v)= \alpha_2 \circ r^2 g(X, Y) +\beta_2 \circ r^2 g(X, \mathcal{U})g(Y, \mathcal{U}), \\
             G(X^v,Y^v)= \alpha_1 \circ r^2 g(X, Y) +\beta_1 \circ r^2 g(X, \mathcal{U})g(Y, \mathcal{U}).
           \end{array}
    \right.
  \end{equation}
\end{definition}

In what follows, we denote by $\phi_i$, $i=1,2,3$, $\alpha$ and $\phi$ the functions defined by:
\begin{itemize}
  \item $\phi_i(t):=\alpha_i(t) +t\beta_i(t)$, for all $t \in ]0,+\infty[$;
  \item $\alpha :=\alpha_1(\alpha_1 +\alpha_3) - \alpha^2$;
  \item $\phi:=\phi_1(\phi_1+\phi_3) -\phi_2^2$.
\end{itemize}

We now turn our attention to the classification of non-degenerate and Riemannian F-natural metrics. A Finslerian F-natural metric is said to be non-degenerate if its associated bilinear form is non-degenerate at each point in $\widetilde{TM}$, and it is Riemannian if this bilinear form is positive definite. The characterization of such metrics is given by the following proposition, whose proof is similar to those of Propositions 2.7 and 2.8 in \cite{AS}:

\begin{proposition}\label{riem-nat}
  An $F$-natural metric $G$ given by \eqref{F-nat-metr} is
  \begin{enumerate}
    \item non-degenerate if and only if $\alpha >0$ and $\phi >0$;
    \item Riemannian if and only if $\alpha_1 >0$, $\alpha >0$, $\phi_1 >0$ and $\phi >0$
  \end{enumerate}
\end{proposition}

As for the case of $g$-natural metrics on the tangent bundle of a Riemannian manifold, the following (subclasses of) $F$-natural metrics on the tangent bundle of a Finsler manifold are interesting:
\begin{itemize}
	\item the {\em Sasaki metric} $g^S$ is obtained for $\alpha _1 =1$ and $\alpha _2 = \alpha _3 = \beta _1 =\beta _2 = \beta _3 =0$;
	the {\em Cheeger-Gromoll metric} is obtained for $\alpha_1(t)=\beta_1(t)=\frac{1}{1+t}$ for all $t \in \mathbb{R}_+$,  $\alpha_2=\beta_2=0,$ $\alpha_1+\alpha_3=1$, $\beta_1+\beta_3=0$;
	\item {\em Kaluza--Klein metrics}, as commonly defined on principal bundles {(see for example \cite{Wood})}, are obtained for
	$\alpha _2 = \beta _2 = \beta _1 +\beta _3 = 0$.
	\item {\em Metrics of Kaluza--Klein type} are defined by the geometric condition of orthogonality between horizontal and vertical distributions. Thus, a  $g$-natural metric $G$ is of Kaluza-Klein type if $\alpha _2=\beta _2 =0$.
\end{itemize}

\subsection{The Levi-Civita Connection of an $F$-natural metric}


\begin{proposition}\label{lev-civ-con}
Let $(M,F,g)$ be a Finsler manifold, $\nabla$ its Chern connection and $R$ the corresponding Riemannian curvature. Let $G$ be a pseudo-Riemannian $F$-natural metric on the slit tangent bundle $\widetilde{TM}$. Then the Levi-Civita connection $\bar \nabla$ of $(\widetilde{TM},G)$ is characterized by $$ \arraycolsep1.5pt
\begin{array}{rcl}
(i) (\bar \nabla_{X^h}Y^h) & = &  h\{(\nabla_{X^h} Y)  +P_{hh}(X,Y)\} + v\{Q_{hh}(X,Y)\}, \\
(ii) (\bar \nabla_{X^h}Y^v) & = &  h\{P_{hv}(X,Y)\} + v\{(\nabla_{X^h} Y)+ Q_{hv}(X,Y)\}, \\
(iii) (\bar \nabla_{X^v}Y^h) & = &  h\{P_{hv}(Y,X)\} + v\{Q_{hv}(Y,X)\}, \\
(iv) (\bar \nabla_{X^v}Y^v) & = &   h\{P_{vv}(X,Y)\} + v\{Q_{vv}(X,Y)\},
\end{array}
\arraycolsep5pt $$ 
for all $X, Y \in \mathfrak{X}(M)$, where $P_{hh}$, $Q_{hh}$, $P_{hv}$, $Q_{hv}$, $P_{vv}$ and $Q_{vv}$ are the $(1,2)$-tensor sections on $\pi_0^*TM$ defined by: $$\arraycolsep1.5pt
\begin{array}{cl}
P_{hh}(s_1,s_2) =  & -\frac{\alpha_1 \alpha_2}{2\alpha} [\mathcal{R}(s_1,\mathcal{U})s_2 +\mathcal{R}(s_2,\mathcal{U})s_1 +2\bar{B}(s_1,s_2)
                 +2\bar{C}(s_2,\mathcal{R}(\mathcal{U},s_1)\mathcal{U})\\
                 & +2\bar{C}(s_1,\mathcal{R}(\mathcal{U},s_2)\mathcal{U})] +  \frac{\alpha_2 (\beta_1 +\beta_3)}{2\alpha}[g(s_2,\mathcal{U})s_1 + g(s_1,\mathcal{U})s_2]  \\
           & +   \frac{1}{\alpha \phi}\{ \alpha_2[\alpha_1(\phi_1 (\beta_1 +\beta_3) -\phi_2 \beta_2) +\alpha_2(\beta_1 \alpha_2 -  \beta_2 \alpha_1)]\mathcal{R}(s_1,\mathcal{U},s_2,\mathcal{U}) \\
           & + \phi_2 \alpha(\alpha_1 +  \alpha_3)^\prime g(s_1,s_2)  +  [\alpha \phi_2(\beta_1 +\beta_3)^\prime  + (\beta_1 +\beta_3)[\alpha_2(\phi_2 \beta_2   \\
           & - \phi_1 (\beta_1 +\beta_3)) +  (\alpha_1 +\alpha_3)(\alpha_1 \beta_2 - \alpha_2 \beta_1)]]g(s_1,\mathcal{U})g(s_2,\mathcal{U}) \}\mathcal{U} \\
           & +  \frac{\alpha_2(\alpha_1+\alpha_3)}{\alpha}\bar{C}(s_1,s_2) +\frac{\alpha_2^2}{\alpha} \bar{L}(s_1,s_2),
\end{array}
\arraycolsep5pt$$

$$\arraycolsep1.5pt
\begin{array}{cl}
Q_{hh}(s_1,s_2) = &   \frac{\alpha_2 ^2}{\alpha}[ \mathcal{R}(s_1,\mathcal{U})s_2 +\bar{B}(s_1,s_2)+ \bar{C}(s_2,\mathcal{R}(\mathcal{U},s_1)\mathcal{U}) +\bar{C}(s_1,\mathcal{R}(\mathcal{U},s_2)\mathcal{U})]\\
& -\frac{\alpha_1 (\alpha_1 +\alpha_3)}{2\alpha} \mathcal{R}(s_1,s_2)\mathcal{U} -  \frac{(\alpha_1 +\alpha_3) (\beta_1 +\beta_3)} {2\alpha}[g(s_2,\mathcal{U})s_1 +g(s_1,\mathcal{U})s_2] \\
         & +  \frac{1}{\alpha\phi} \{ \alpha_2[\alpha_2(\phi_2 \beta_2- \phi_1 (\beta_1 +\beta_3)) +(\alpha_1 +\alpha_3)(\beta_2 \alpha_1  \\
         & -  \beta_1 \alpha_2)]\mathcal{R}(X,\mathcal{U},Y,\mathcal{U}) -\alpha (\phi_1 +\phi_3)(\alpha_1+ \alpha_3)^\prime g(s_1,s_2)\\
         & +  [-\alpha (\phi_1 +\phi_3)(\beta_1 +\beta_3)^\prime+ (\beta_1 +\beta_3)[(\alpha_1 +\alpha_3)[(\phi_1 +\phi_3) \beta_1- \phi_2 \beta_2]  \\
         & +  \alpha_2[\alpha_2 (\beta_1 +\beta_3) -(\alpha_1 +\alpha_3) \beta_2 ]]g(s_1,\mathcal{U})g(s_2,\mathcal{U}) \}\mathcal{U}\\
         & -  \frac{(\alpha_1+\alpha_3)^2}{\alpha}\bar{C}(s_1,s_2) -\frac{\alpha_2(\alpha_1+\alpha_3)}{\alpha} \bar{L}(s_1,s_2),
\end{array}
\arraycolsep5pt$$

$$\arraycolsep1.5pt
\begin{array}{cl}
P_{hv}(s_1,s_2)  = &  -\frac{\alpha_1^2}{2 \alpha} [\mathcal{R}(s_2,\mathcal{U})s_1 +\bar{B}(s_1,s_2)+\bar{C}(s_1,\mathcal{R}(\mathcal{U},s_2)\mathcal{U})+\bar{C}(s_2,\mathcal{R}(\mathcal{U},s_1)\mathcal{U})] \\
&-\frac{\alpha_1 (\beta_1 +\beta_3)}{2 \alpha} g(s_2,\mathcal{U})s_1 +  \frac{1}{\alpha}[\alpha_1 (\alpha_1 +\alpha_3)^\prime -\alpha_2(\alpha_2 ^\prime -\frac{\beta_2}{2})]  g(s_1,\mathcal{U})s_2 \\
         & +  \frac{1}{\alpha \phi} \{ \frac{\alpha_1}2 [\alpha_2(\alpha_2 \beta_1 -\alpha_1 \beta_2) +\alpha_1 (\phi_1 (\beta_1 +\beta_3) -  \phi_2 \beta_2)]\mathcal{R}(s_2,\mathcal{U},s_1,\mathcal{U})\\
         & + \alpha [\frac{\phi_1}2 (\beta_1 +\beta_3)
               +\phi_2(\alpha_2 ^\prime -\frac{\beta_2}{2})] g(s_1,s_2) +  [\alpha \phi_1  (\beta_1 +\beta_3)^\prime\\
         & + [\alpha_2(\alpha_1 \beta_2 -\alpha_2 \beta_1)+  \alpha_1(\phi_2 \beta_2 -(\beta_1 +\beta_3) \phi_1)] [(\alpha_1 +\alpha_3)^\prime +\frac{\beta_1 +\beta_3}2]\\
         & +  [\alpha_2 (\beta_1 (\phi_1 +\phi_3) -\beta_2 \phi_2) -\alpha_1 (\beta_2 (\alpha_1 +\alpha_3) \\
         & -  \alpha_2(\beta_1 +\beta_3)](\alpha_2 ^\prime  -\frac{\beta_2}{2})]g(s_1,\mathcal{U})g(s_2,\mathcal{U}) \}\mathcal{U} \\
         & +  \frac{\alpha_1(\alpha_1+\alpha_3)}{\alpha}\bar{C}(s_1,s_2) +\frac{\alpha_1\alpha_2}{\alpha} \bar{L}(s_1,s_2),
\end{array}
\arraycolsep5pt$$

$$\arraycolsep1.5pt
\begin{array}{cl}
Q_{hv}(s_1,s_2)  = & \frac{1}{\alpha} \{\frac{\alpha_1 \alpha_2}2 [\mathcal{R}(s_2,\mathcal{U})s_1 +\bar{B}(s_1,s_2)+\bar{C}(s_1,\mathcal{R}(\mathcal{U},s_2)\mathcal{U})+\bar{C}(s_2,\mathcal{R}(\mathcal{U},s_1)\mathcal{U})] \\ 
& +\frac{\alpha_2 (\beta_1 +\beta_3)}{2} g(s_1,\mathcal{U})s_2 +  [-\alpha_2 (\alpha_1 +\alpha_3)^\prime +(\alpha_1 +\alpha_3) (\alpha_2 ^\prime -\frac{\beta_2}{2})] g(s_2,\mathcal{U})s_1\}\\
         & +   \frac{1}{\alpha \phi} \{ \frac{\alpha_1}2 [(\alpha_1 +\alpha_3)(\alpha_1 \beta_2 - \alpha_2 \beta_1) +  \alpha_2 (\phi_2 \beta_2 -\phi_1 (\beta_1 +\beta_3))]\mathcal{R}(s_1,\mathcal{U},s_2,\mathcal{U})  \\
         & -  \alpha [\frac{\phi_2}2 (\beta_1 +\beta_3) +(\phi_1 +\phi_3)(\alpha_2 ^\prime -\frac{\beta_2}{2})] g(s_1,s_2) \\
         & +  [-\alpha \phi_2 (\beta_1 +\beta_3)^\prime +[(\alpha_1 +\alpha_3)(\alpha_2 \beta_1 -\alpha_1 \beta_2) \\
         & +  \alpha_2(\phi_1(\beta_1 +\beta_3) -\phi_2 \beta_2)] [(\alpha_1 +\alpha_3)^\prime +\frac{\beta_1 +\beta_3}2] \\
         & +  [(\alpha_1 +\alpha_3)(\beta_2 \phi_2 -\beta_1 (\phi_1 +\phi_3)) +\alpha_2 (\beta_2 (\alpha_1 +\alpha_3) \\
         & -  \alpha_2(\beta_1 +\beta_3)](\alpha_2 ^\prime -\frac{\beta_2}{2})]g(s_1,\mathcal{U})g(s_2,\mathcal{U}) \}\mathcal{U} \\
         & -  \frac{\alpha_2(\alpha_1+\alpha_3)}{\alpha}\bar{C}(s_1,s_2) -\frac{\alpha_1(\alpha_1+\alpha_3)}{\alpha} \bar{L}(s_1,s_2),
\end{array}
\arraycolsep5pt$$

$$\arraycolsep1.5pt
\begin{array}{cl}
P_{vv}(s_1,s_2)  = &  \frac{1}{\alpha}[\alpha_1 (\alpha_2^\prime +\frac{\beta_2}{2})-\alpha_2 \alpha_1^\prime] [g(s_2,\mathcal{U})s_1 + g(s_1,\mathcal{U})s_2]\\
         & +  \frac{1}{\alpha \phi} \{\alpha [\phi_1 \beta_2 -\phi_2(\beta_1 -\alpha_1 ^\prime)]g(s_1,s_2)  \\
         & +  [\alpha (2\phi_1 \beta_2^\prime -\phi_2 \beta_1^\prime) +2\alpha_1^\prime [\alpha_1(\alpha_2 (\beta_1 +\beta_3) \\
         & -  \beta_2 (\alpha_1 +\alpha_3)) +\alpha_2(\beta_1(\phi_1 +\phi_3) -\beta_2 \phi_2)] \\
         & +  (2\alpha_2 ^\prime +\beta_2) [\alpha_1(\phi_2 \beta_2 -\phi_1(\beta_1 +\beta_3))\\
         & +  \alpha_2 (\alpha_1 \beta_2 -\alpha_2 \beta_1)]]g(s_1,\mathcal{U})g(s_2,\mathcal{U}) \}\mathcal{U} \\
         & +  \frac{\alpha_1\alpha_2}{\alpha}\bar{C}(s_1,s_2) +\frac{\alpha_1^2}{\alpha} \bar{L}(s_1,s_2),
\end{array}
\arraycolsep5pt$$

$$\arraycolsep1.5pt
\begin{array}{cl}
Q_{vv}(s_1,s_2)  = &  \frac{1}{\alpha} [-\alpha_2 (\alpha_2^\prime +\frac{\beta_2}{2}) +(\alpha_1 +\alpha_3) \alpha_1^\prime]
               [g(s_2,\mathcal{U})s_1 + g(s_1,\mathcal{U})s_2]\\
         & +  \frac{1}{\alpha \phi} \{\alpha [(\phi_1 +\phi_3)(\beta_1 -\alpha_1 ^\prime) -\phi_2 \beta_2]g(s_1,s_2)\\
         & +  [\alpha ((\phi_1 +\phi_3) \beta_1^\prime -2\phi_2 \beta_2^\prime) +2\alpha_1^\prime [\alpha_2 (\beta_2 (\alpha_1 +\alpha_3)\\
         & -  \alpha_2 (\beta_1 +\beta_3))+ (\alpha_1 +\alpha_3) (\beta_2 \phi_2 -\beta_1(\phi_1+\phi_3))] \\
         & +  (2\alpha_2 ^\prime +\beta_2) [\alpha_2(\phi_1(\beta_1 +\beta_3) -\phi_2 \beta_2) \\
         & +  (\alpha_1 +\alpha_3)(\alpha_2\beta_1 -\alpha_1 \beta_2)]]g(s_1,\mathcal{U})g(s_2,\mathcal{U}) \}\mathcal{U} \\
         & +  \frac{\alpha-\alpha_2^2}{\alpha}\bar{C}(s_1,s_2) -\frac{\alpha_1\alpha_2}{\alpha} \bar{L}(s_1,s_2),
\end{array}
\arraycolsep5pt$$
where all the functions $\alpha_i$, $\beta_i$, $\phi_i$, $\alpha$ and $\phi$ are taken composed by $F^2$.
\end{proposition}


\subsection{Incompressibility of the Geodesic flow}


The geodesic vector field plays a crucial role in understanding the dynamics of the manifold, and incompressibility is an important property related to the conservation of volume along the geodesic flow. In this section, we prove that the geodesic flow on the slit tangent bundle of a Finsler manifold equipped with a pseudo-Riemannian $F$-natural metric is incompressible, generalizing a similar result when the base manifold is a Riemannian manifold and the tangent bundle is a $g$-natural metric (cf. \cite{Abb-Cal-Per-HJM}). 

\begin{theorem}\label{g-f-incomp}
  Let $(M,F,g)$ be a Finsler manifold and $G$ be a pseudo-Riemannian $F$-natural metric on the slit tangent bundle $\widetilde{TM}$. Then the geodesic flow vector field $\zeta$ on $\widetilde{TM}$ is always incompressible with respect to $G$.
\end{theorem}

\begin{proof}
  Using the notations of Proposition \ref{lev-civ-con}, we can check easily that the divergence $(\textup{div}_G \zeta)_{(x,u)}$ of $\zeta$ at an arbitrary point $(x,u) \in \widetilde{TM}$, with respect to $G$, is the trace of the endomorphism of $M_x$ given by $X \rightarrow P_{hh}(u,X)+Q_{hv}(u,X)$. Now, by simple calculation, for all $X \in \mathfrak{X}(M)$, we have
\begin{equation*}
\begin{split}
 P_{hh}(\mathcal{U},X)=& -Q_{hv}(\mathcal{U},X)\\
 =& \frac{1}{2\alpha}\left\{-\alpha_1 \alpha_2 \mathcal{R}(X,\mathcal{U})\mathcal{U} +  \alpha_2 (\beta_1 +\beta_3). r^2 X +\left[-\alpha_2 (\beta_1 +\beta_3) \frac{}{} \right.\right. \\
  & \left.\left. + \frac{2\alpha(\beta_1 +\beta_3)\phi_2}{\phi} [(\alpha_1+\alpha_3)^\prime+ (\beta_1  +\beta_3)+ (\beta_1 +\beta_3)^\prime. r^2]g(X,\mathcal{U})\mathcal{U}  \right]\right\},
\end{split}
\end{equation*}
Hence $(\textup{div}_G \zeta)_{(x,u)}=0$, for all $(x,u) \in \widetilde{TM}$. We deduce that $\zeta$ is incompressible on $(\widetilde{TM},G)$.
\end{proof}


\subsection{When are the fibers of the slit tangent bundle totally geodesic?}


In this section, we address the conditions under which the fibers of the slit tangent bundle $\widetilde{TM}$ of a Finsler manifold $(M,F,g)$ are totally geodesic with respect to a pseudo-Riemannian $F$-natural metric. Recall that a submanifold is said to be totally geodesic if any geodesic that starts tangent to the submanifold remains within it. In the context of the slit tangent bundle, the fibers correspond to the punctured tangent spaces $T_xM\setminus \{0\}$ for each $x \in M$, and we seek to determine when these fibers are totally geodesic in $\widetilde{TM}$. This leads to a classication of pseudo-Riemannian $F$-natural metrics on $\widetilde{TM}$ which possesses this property.

\begin{theorem}\label{t-g-fib}
Let $(M,F,g)$ be a Finsler manifold and $G$ be a pseudo-Riemannian $F$-natural metric on the slit tangent bundle $\widetilde{TM}$. The fibers of $(\widetilde{TM},G)$ are totally geodesic if and only if the two following assertions hold 
\begin{enumerate}
  \item $\tilde{L}=\frac{\alpha_2}{\alpha_1}\tilde{C}$;
  \item there is a real constant $c$ such that
\begin{equation}\label{tot-geo-fib}
\left\{
\begin{array}{l}
\alpha_2(t)=\frac{c}{\sqrt{|\phi_1(t)|}}(t. \alpha_1^\prime(t) +\alpha_1(t)), \\ 
\beta_2(t)=\frac{c}{\sqrt{|\phi_1(t)|}}(\beta_1(t) -\alpha_1^\prime(t)),
\end{array}   \right.
\end{equation}
for all $t >0$.
\end{enumerate}
\end{theorem}

\begin{Proof}
Remark first that the fibers of $(\widetilde{TM},G)$ are totally geodesic if and only if $\bar \nabla _{X^v}X^v $ is vertical, for all $X \in \mathfrak{X}(M)$ (cf. \cite{Bes}, p.47). Hence, by virtue of Proposition \ref{lev-civ-con}, the fibers of $(TM,G)$ are totally geodesic if and only if $P_{vv}(X,X)=0$, for all $X \in \mathfrak{X}(M)$. Since $E$ is symmetric and linear in the second and third arguments, the last assertion is equivalent to
$[P_{vv}(X,X)]_{(x,u)}=0$, for all $X \in \mathfrak{X}(M)$ and $(x,u) \in \widetilde{TM}$. 

But, if $u \bot X_x$ then we have by virtue of Proposition \ref{lev-civ-con},
\begin{equation*}
\begin{split}
  [P_{vv}(X,X)]_{(x,u)}= & \frac{1}{\phi(r^2)} (\phi_1 \beta_2 -\phi_2 (\beta_1 -\alpha_1^\prime))(r^2).g(X_x,X_x).u  \\
    & +  \frac{\alpha_1(r^2)\alpha_2(r^2)}{\alpha(r^2)}\bar{C}(X_x,X_x) +\frac{\alpha_1^2(r^2)}{\alpha(r^2)} \bar{L}(X_x,X_x).
\end{split}
\end{equation*}
Hence, $[P_{vv}(X,X)]_{(x,u)}=0$, for all $(x,u) \in \widetilde{TM}$ such that $u \bot X_x$, is equivalent to
\begin{equation}\label{fib010}
\begin{split}
  \frac{1}{\phi(r^2)} (\phi_1 \beta_2& -\phi_2 (\beta_1 -\alpha_1^\prime))(r^2).g(X_x,X_x).u =  \\
    & =-  \frac{\alpha_1(r^2)\alpha_2(r^2)}{\alpha(r^2)}\bar{C}(X_x,X_x) -\frac{\alpha_1^2(r^2)}{\alpha(r^2)} \bar{L}(X_x,X_x).
\end{split}  
\end{equation}
Making the scalar product of the last identity with $u$, we find
\begin{equation*}
\begin{split}
  \frac{1}{\phi(r^2)} (\phi_1 \beta_2 &-\phi_2 (\beta_1 -\alpha_1^\prime))(r^2)r^2.g(X_x,X_x)=\\  =&-\frac{\alpha_1(r^2)\alpha_2(r^2)}{\alpha(r^2)}g(\bar{C}(X_x,X_x),u) -\frac{\alpha_1^2(r^2)}{\alpha(r^2)} g(\bar{L}(X_x,X_x),u) \\
     =& \frac{\alpha_1(r^2)\alpha_2(r^2)}{\alpha(r^2)}C(X_x,X_x,u) -\frac{\alpha_1^2(r^2)}{\alpha(r^2)} L(X_x,X_x,u) =0.
\end{split}
\end{equation*}
We deduce that 
\begin{equation}\label{fib01}
\phi_1 \beta_2 =\phi_2 (\beta_1 -\alpha_1^\prime),
\end{equation}
on $\mathbb{R}^{+*}$. Now, by virtue of \eqref{fib01}, \eqref{fib010} becomes
\begin{equation*}
  \frac{\alpha_1(r^2)\alpha_2(r^2)}{\alpha(r^2)}\bar{C}_u(X_x,X_x) =-\frac{\alpha_1^2(r^2)}{\alpha(r^2)} \bar{L}_u(X_x,X_x),
\end{equation*}
for all $u \bot X_x$. For $u$ linear to $X_x$, the last identity is obviously satisfied since $\bar{C}_u(u,.)=\bar{L}_u(u,.)=0$.

On the other hand, we have for all $(x,u) \in \widetilde{TM}$,
$$
\arraycolsep1.5pt
\begin{array}{cl}
[P_{vv}(u,u)]_{(x,u)}  = &  \frac{r^2}{\phi} \{\phi_1 \beta_2
                -\phi_2 (\beta_1 -\alpha_1^\prime)
                +\frac{1}{\alpha}\{2 \phi [\alpha_1
                (\alpha_2^\prime +\frac{\beta_2}{2}) -\alpha_2
               \alpha_1^\prime]\\
         & +  \alpha [2\phi_1 \beta_2^\prime
               -\phi_2 \beta_1^\prime].r^2 +2\alpha_1^\prime. r^2[\alpha_1
               (\alpha_2 (\beta_1 +\beta_3)-\beta_2 (\alpha_1 +\alpha_3))\\
         & +   \alpha_2(\beta_1(\phi_1 +\phi_3) -\beta_2 \phi_2)]
               \\
         & +(2\alpha_2 ^\prime +\beta_2).r^2 [\alpha_1(\phi_2 \beta_2
               -\phi_1(\beta_1 +\beta_3))+  \alpha_2 (\alpha_1 \beta_2 -\alpha_2 \beta_1)]
               \}\}u  \\
         & =  \frac{r^2}{\phi} \{\phi_1 \beta_2
                -\phi_2 (\beta_1 -\alpha_1^\prime)
                +\frac{1}{\alpha}\{\alpha [2\phi_1
                \beta_2^\prime -\phi_2 \beta_1^\prime].r^2 \\
         & +  2 \alpha_1^\prime [
                \alpha_2(-\phi +\alpha_1 (\beta_1
                +\beta_3).r^2 + (\phi_1 +\phi_3) \beta_1 .r^2
                - \phi_2 \beta_2.r^2) \\
         & -  \alpha_1 (\alpha_1 +\alpha_3)\beta_2.r^2]
                + (2 \alpha_2^\prime +\beta_2)
                [\alpha_1 (\phi +(\phi_2 \beta_2 \\
         & -  \phi_1(\beta_1 +\beta_3)) .r^2)
               +\alpha_2(\alpha_1 \beta_2 -\alpha_2 \beta_1).r^2 ]\}\}u ,
\end{array}
\arraycolsep5pt$$
where $r^2=g(u,u)$. But
$$
\arraycolsep1.5pt
\begin{array}{cl}
  & \alpha_2(-\phi +\alpha_1 (\beta_1 +\beta_3).r^2 + (\phi_1
    +\phi_3) \beta_1 .r^2 - \phi_2 \beta_2.r^2)
    -\alpha_1 (\alpha_1 +\alpha_3)\beta_2.r^2\\
= & \alpha_2[(\phi_2^2 -\phi_2 \beta_2.r^2) + ((\phi_1 +\phi_3)
    \beta_1 .r^2 -\phi_1(\phi_1 +\phi_3))+\alpha_1 (\beta_1
    +\beta_3).r^2]\\
  & -\alpha_1 (\alpha_1 +\alpha_3)\beta_2.r^2 \\
= & \alpha_2[\phi_2 \alpha_2 - \alpha_1(\phi_1 +\phi_3)
    +\alpha_1 (\beta_1 +\beta_3).r^2]-\alpha_1 (\alpha_1
    +\alpha_3)\beta_2.r^2\\
= & \alpha_2[\phi_2 \alpha_2 - \alpha_1 (\alpha_1 +\alpha_3)]
    -\alpha_1 (\alpha_1 +\alpha_3)\beta_2.r^2 \\
= & \alpha_2^2 \phi_2 -\alpha_1 (\alpha_1 +\alpha_3)(\alpha_2
    +\beta_2.r^2) \\
= & -\alpha .\phi_2.
\end{array}
\arraycolsep5pt$$
By similar way, we find that
$$
\alpha_1 (\phi +(\phi_2 \beta_2 -\phi_1(\beta_1
    +\beta_3)).r^2) +\alpha_2(\alpha_1 \beta_2 -\alpha_2 \beta_1).r^2
=  \alpha. \phi_1,
$$
so that, we obtain
$$
\arraycolsep1.5pt
\begin{array}{cl}
[P_{vv}(u,u)]_{(x,u)}  = &  \frac{r^2}{\phi} \{\phi_1 \beta_2
                -\phi_2 (\beta_1 -\alpha_1^\prime)
                +(2\phi_1\beta_2^\prime -\phi_2 \beta_1^\prime).r^2 \\
         &     - 2 \phi_2 \alpha_1^\prime
                + \phi_1 (2 \alpha_2^\prime +\beta_2)\}u \\
         & =  \frac{r^2}{\phi} \{2\phi_1 (\beta_2 +\alpha_2^\prime
                +\beta_2^\prime.r^2) -\phi_2 (\beta_1 +\alpha_1^\prime
                +\beta_1^\prime.r^2)\}u   \\
         & =  \frac{r^2}{\phi} \{2\phi_1 \phi_2^\prime
               -\phi_2 \phi_1^\prime \}u.
\end{array}
\arraycolsep5pt$$
Hence, $[P_{vv}(u,u)]_{(x,u)}=0$, for all $(x,u) \in \widetilde{TM}$, if and only if $2\phi_1 \phi_2^\prime -\phi_2 \phi_1^\prime=0$ on $\mathbb{R}^{+*}$.
We deduce that the fibers of $(\widetilde{TM},G)$ are totally geodesic if and only if
\begin{equation}\label{fib1}
\left\{
\begin{array}{l}
\tilde{L}=\frac{\alpha_2}{\alpha_1}\circ r^2\tilde{C}\\
\phi_1 \beta_2 =\phi_2 (\beta_1 -\alpha_1^\prime), \\ 
2\phi_1 \phi_2^\prime =\phi_2 \phi_1^\prime,
\end{array}   \right.
\end{equation}
on $\mathbb{R}^+_*$. Now, the last identity of \eqref{fib1} is equivalent, by virtue of the second identity, to
\begin{equation}\label{fib2}
\left\{
\begin{array}{l}
\alpha_2(t) =\beta_2(t) =0 \: \mbox{ whenever } \: \phi_2(t)=0, \\
\phi_1=d.\phi_2^2 \: \mbox{on each interval where } \:
\phi_2(t) \neq 0 \mbox{ everywhere,}
\end{array}   \right.
\end{equation}
where $d$ is a real constant of the same sign as $\phi_1$.

Denote by $J$ the complement of $\phi_2^{-1}(0)$ in $\mathbb{R}^+_*$. $J$ is an open subset of $\mathbb{R}^+_*$. We claim that, in the conditions of (\ref{fib2}), either $J =\emptyset$ or $J =\mathbb{R}^+_*$. If not, there is $0< a< b$ such that $]a,b[ \subset J$ (since $J$ is open) and $a \not \in
J$. Then there is a constant $d>0$ such that $\phi_1 =d \cdot \phi_2^2$ on $]a,b[$. When $t \rightarrow a$, we have by continuity of $\phi_1$ and $\phi_2$, $\phi_1(a) =d \cdot \phi_2^2(a)$. Since $a \not \in J$, then $\phi_1(a) =0$ and consequently $\phi(a)=\phi_1(a)(\phi_1+\phi_3)(a) -\phi_2^2(a)=0$, which contradicts the fact that $G$ is pseudo-Riemannian (Proposition \ref{riem-nat}). We deduce that either $J =\emptyset$ or $J =\mathbb{R}^+_*$. Hence, (\ref{fib2}) is equivalent to
\begin{equation*}
\left\{
\begin{array}{cl}
\mbox{ either } & \alpha_2 =\beta_2 =0 \: \mbox{ on } \: \mathbb{R}^+_*, \\ 
\mbox{ or } & \phi_1=d.\phi_2^2 \: \mbox{ on } \:\mathbb{R}^+_*.
\end{array}   \right.
\end{equation*}
When $d \neq 0$, we have $\phi_1 \neq 0$, everywhere on $\mathbb{R}^+_*$, since otherwise there is $t_0 >0$ such that $\phi_1(t_0)=\phi_2(t_0)=0$ and then $\phi(t_0)=0$, which contradicts the fact that $G$ is pseudo-Riemannian. In this case the first equation of (\ref{fib1}) is equivalent to 
\begin{equation}\label{fib41}
  \beta_2= \frac{c}{\sqrt{|\phi_1|}} (\beta_1 -\alpha_1^\prime),
\end{equation}
where $c= \frac{\pm1}{\sqrt{|d|}}$. Using the fact that $\alpha_2=\phi_2 -t\beta_2$, we obtain
\begin{equation}\label{fib42}
  \alpha_2=  \frac{c}{\sqrt{|\phi_1|}} (\alpha_1+ t\alpha_1^\prime).
\end{equation}
The case $d=0$ gives the same expressions \eqref{fib41} and \eqref{fib42} of $\beta_2$ and $\alpha_2$, respectively, with $c=0$. This completes the proof of the theorem.
\end{Proof}

Note that $c=0$, in the system (\ref{tot-geo-fib}), corresponds to the $g$-natural metrics on $\widetilde{TM}$ of Kaluza-Klein type. This gives the following characterization of Landsberg manifolds in terms of geometric properties of their slit tangent bundles:
\begin{corollary}
  Let $(M,F,g)$ be a Finsler manifold and $G$ be a pseudo-Riemannian $F$-natural metric of Kaluza-Klein type on the slit tangent bundle $\widetilde{TM}$. The fibers of $(\widetilde{TM},G)$ are totally geodesic if and only if $(M,F,g)$ is a Landsberg manifold.
\end{corollary}


\section{Symmetries on slit tangent bundles of Finsler manifolds}


By symmetry we mean a transformation preserving some natural properties of a given space. Given a pseudo-Riemannian manifold $(M,g)$ and a tensor $T$ of $(M,g)$, codifying some mathematical or physical quantity, a symmetry of $T$ is a one-parameter group of diffeomorphisms of $(M,g)$, which leaves $T$ invariant. As such, it corresponds to a vector field $X$ satisfying $\mathcal{L}_X T=0$, where $\mathcal{L}$ denotes the Lie derivative. Among the
symmetries of a given manifold $(M,g)$, particularly relevant examples are given by isometries, homotheties and conformal motions. A smooth vector field $\xi$ on a pseudo-Riemannian manifold $(M,g)$ is said to be conformal if there exists a smooth function $f$ on $M$, such that
\begin{equation*}
  \mathcal{L}_\xi g=2fg,
\end{equation*}
that is, the flow of the vector field $\xi$ consists of conformal transformations of the pseudo-Riemannian manifold $(M,g)$. The function $f$ is called the potential function of the conformal vector field $\xi$. When $f$ is constant (respectively, $f=0$), the flow of $\xi$ is given by homothetic (respectively, isometric) transformations of $(M,g)$, and $\xi$ is called an homothetic (respectively, Killing) vector field. The study of the symmetries of a given pseudo-Riemannian manifold enriches our understanding of its geometric features. We may refer to the monograph [11] for a nice and extensive introduction on symmetries, and to [5,6] and references therein for recent examples of investigations of symmetries in some given classes of homogeneous pseudo-Riemannian manifolds.

In this section, we will investigate when some special vector fields on $\widetilde{TM}$ are conformal, homothetic or Killing with respect to an arbitrary pseudo-Riemannian $F$-natural metric of Kaluza-Klein type. 


\subsection{Horizontal lift vector fields}


\begin{lemma}\label{Lem-Lie-hor}
Let $(M,F,g)$ be a Finsler manifold and $G$ be a pseudo-Riemannian $F$-natural metric on the slit tangent bundle $\widetilde{TM}$. Let $\xi$ be a vector field on $M$. The Lie derivative of $G$ along $\xi ^{h}$ is given by:
$$\arraycolsep1.5pt
\begin{array}{rl}
\mathcal{L}_{\xi ^{h}}G(X^{h},Y^{h})=& (\alpha _{1}+\alpha _{3})(r{{}^2})[g(\nabla _{X^{h}}\xi ,Y)+g(\nabla _{Y^{h}}\xi ,X)] \\ 
&+(\beta _{1}+\beta _{3})(r{{}^2})[g(\nabla _{X^{h}}\xi, \mathcal{U})g(Y,\mathcal{U})+g(\nabla _{Y^{h}}\xi ,\mathcal{U})g(X,\mathcal{U})] \\ 
&-\alpha _{2}\left( r{{}^2}\right) [\mathcal{R}(X,\mathcal{U},\xi ,Y)+\mathcal{R}(Y,\mathcal{U},\xi ,X)+2C(\xi ,\mathcal{R}(\mathcal{U},X)\mathcal{U},Y) \\ 
&+2C(\xi ,\mathcal{R}(\mathcal{U},Y)\mathcal{U},X)+2C(X,\mathcal{R}(\mathcal{U},\xi )\mathcal{U},Y)] \\ 
& \\
\mathcal{L}_{\xi ^{h}}G(X^{h},Y^{v})=&\alpha _{2}(r{{}^2})[g(\nabla _{X^{h}}\xi ,Y)-L(X,Y,\xi )]+\beta _{2}(r{{}^2})g(\nabla _{X^{h}}\xi, \mathcal{U})g(Y,\mathcal{U}) \\ 
&+\alpha _{1}(r{{}^2})[\mathcal{R}(\mathcal{U},Y,\xi ,X)+B(Y,\xi ,\mathcal{U},X)] \\
&\\
\mathcal{L}_{\xi ^{h}}G(X^{v},Y^{v})=&-2\alpha _{1}(r{{}^2})L(X,Y,\xi )
\end{array}$$
\end{lemma}

\begin{theorem}\label{Th-Lie-hor}
  Let $(M,F,g)$ be a Finsler manifold and $G$ be a pseudo-Riemannian $g$-natural metric on $\widetilde{TM}$ of Kaluza Klein type. Given an arbitrary vector field $\xi $ on $M$, the following statements are equivalent
\begin{enumerate}
    \item $\xi ^{h}$ is a conformal vector field on $(\widetilde{TM},G)$;
    \item $\xi ^{h}$ is a killing vector field on $(\widetilde{TM},G)$;
    \item the following conditions hold:
    \begin{itemize}
        \item[(i)] $g(\nabla _{X^{h}}\xi ,Y)+g(\nabla _{Y^{h}}\xi ,X)=0$, for any $X, Y \in \mathfrak{X}(M)$,
        \item[(ii)] $2\mathcal{R}(\mathcal{U},Y,\xi ,X)+\mathcal{R}(\mathcal{U},X,\xi ,Y)+\mathcal{R}(\mathcal{U},X,Y,\xi )=0$,
        \item[(iii)] $L(\xi ,.,.)=0$.
    \end{itemize}
\end{enumerate}
\end{theorem}

\begin{proof}
Considering in Lemma \ref{Lem-Lie-hor} that $G$ is of Kaluza-Klein type (i.e. $\alpha_2=\beta_2=0$), then the Lie derivative of $G$ along $\xi ^{h}$ is given by the system

$$\left\{
\begin{array}{l}
\mathcal{L}_{\xi ^{h}}G(X^{h},Y^{h})=(\alpha _{1}+\alpha _{3})(r{{}^2})[g(\nabla _{X^h}\xi ,Y)+g(\nabla _{Y^h}\xi ,X)], \\ 
\\
\mathcal{L}_{\xi ^{h}}G(X^{h},Y^{v})=\alpha _{1}(r{{}^2})[\mathcal{R}(\mathcal{U},Y,\xi ,X)+B(Y,\xi ,\mathcal{U},X)], \\ \\
\mathcal{L}_{\xi ^{h}}G(X^{v},Y^{v})=-2\alpha _{1}(r{{}^2})L(X,Y,\xi ).
\end{array}\right.
$$

If we suppose that $\xi ^{h}$ is conformal on $(\widetilde{TM},G)$, with potential function $\theta$, then the preceding system becomes

$$\left\{
\begin{array}{l}
2\theta (\alpha _{1}+\alpha _{3})(r^{2})g(X,Y)=(\alpha_{1}+\alpha _{3})(r{{}^2})[g(\nabla _{X^h}\xi ,Y)+g(\nabla _{Y^h}\xi ,X)], \\ \\
0=\mathcal{R}(\mathcal{U},Y,\xi ,X)+B(Y,\xi ,\mathcal{U},X), \\  \\
2\theta [\alpha _{1}(r^{2})g(X,Y)+\beta_{1}(r^{2})g(X,\mathcal{U})g(Y,\mathcal{U})]=-2\alpha _{1}(r{{}^2})L(X,Y,\xi ).
\end{array}\right.
$$

Let $(x,u) \in \widetilde{TM}$. If we take $X$ and $Y$ such that $X_x=Y_x=u$, then the preceding system yields

$$\left\{
\begin{array}{l}
\theta (x,u)r^{2}(\alpha _{1}+\alpha _{3})(r^{2})=(\alpha_{1}+\alpha _{3})(r{{}^2})g(\nabla _{u^h}\xi ,u) \\ \\
2r^{2}\phi _{1}(r^{2})\theta (x,u)=0%
\end{array}\right.
$$

Since $\phi _{1}$ and $(\alpha _{1}+\alpha _{3})$ don't vanish, we deduce that $\theta (x,u)=0$. Since $(x,u)$ is arbitrary, then $\theta$ vanishes identically. Hence $\xi^h$ is Killing and the equivalence $1. \Longleftrightarrow 2.$ is proved. To prove the equivalence $2. \Longleftrightarrow 3.$, it suffice to remark that $\xi^h$ is Killing if and only if
$$\left\{
\begin{array}{l}
0=(\alpha_{1}+\alpha _{3})(r{{}^2})[g(\nabla _{X^h}\xi ,Y)+g(\nabla _{Y^h}\xi ,X)], \\ \\
0=\mathcal{R}(\mathcal{U},Y,\xi ,X)+B(Y,\xi ,\mathcal{U},X), \\ \\
0=-2\alpha _{1}(r{{}^2})L(X,Y,\xi ).
\end{array}\right.
$$ 
This completes the proof of the theorem.
\end{proof}


\subsection{Vertical lifts of vector fields}


\begin{lemma}\label{Lem-Lie-ver}
Let $(M,F,g)$ be a Finsler manifold and $G$ be a pseudo-Riemannian $F$-natural metric on the slit tangent bundle $\widetilde{TM}$. Let $\xi$ be a vector field on $M$. The Lie derivative of $G$ along $\xi ^{v}$ is given by:
$$\arraycolsep1.5pt
\begin{array}{rl}
\mathcal{L}_{\xi ^{v}}G(X^{h},Y^{h})=& \alpha _{2}(r{{}^2})(g(\nabla _{X^{h}}\xi ,Y)+g(\nabla _{Y^{h}}\xi ,X)+2L(X,Y,\xi )) \\ 
&+\beta _{2}(r{{}^2})[g(\nabla _{X^{h}}\xi, \mathcal{U})g(Y,\mathcal{U})+g(\nabla _{Y^{h}}\xi ,\mathcal{U})g(X,\mathcal{U})]\\
&+2(\beta_{1}+\beta _{3})^{\prime }(r{{}^2})g(X,\mathcal{U})g(Y,\mathcal{U})g(\xi,\mathcal{U}) \\ 
&+2(\alpha _{1}+\alpha _{3})^{\prime }(r{{}^2})g(X,Y)g(\xi,\mathcal{U})+2(\alpha _{1}+\alpha _{3})(r{{}^2})C(X,Y,\xi ) \\ 
&+(\beta _{1}+\beta _{3})(r{{}^2})\left( g(X,\xi )g(Y,\mathcal{U})+g(Y,\xi )g(X,\mathcal{U})\right) \\
&+\alpha _{1}(r{{}^2})(2B(X,Y,\mathcal{U},\xi )-B(X,\xi ,Y,\mathcal{U})-B(Y,\xi ,X,\mathcal{U})) \\  
& \\
\mathcal{L}_{\xi ^{v}}G(X^{h},Y^{v})=&\alpha _{1}(r{{}^2})(g(\nabla _{X^{h}}\xi ,Y)+L(X,Y,\xi ))+\beta _{1}(r{{}^2})g(\nabla _{X^{h}}\xi, \mathcal{U})g(Y,\mathcal{U})  \\ 
&+2\alpha _{2}^{\prime }(r{{}^2})g(X,Y)g(\xi,\mathcal{U})+\beta _{2}(r{{}^2})(g(Y,\xi )g(X,\mathcal{U})+g(X,\xi )g(Y,\mathcal{U}) \\ 
&+2\alpha _{2}(r{{}^2})C(X,Y,\xi )+2\beta _{2}^{\prime }(r{{}^2})g(X,\mathcal{U})g(Y,\mathcal{U})g(\xi,\mathcal{U})\\
&\\
\mathcal{L}_{\xi ^{v}}G(X^{v},Y^{v})=&2\alpha _{1}^{\prime }(r{{}^2})g(X,Y)g(\xi,\mathcal{U})+2\alpha _{1}(r{{}^2})C(X,Y,\xi ) \\ 
& +\beta _{1}(r{{}^2})(g(X,\xi )g(Y,\mathcal{U})+g(Y,\xi )g(X,\mathcal{U}))\\
&+2\beta _{1}^{\prime }(r{{}^2})g(X,\mathcal{U})g(Y,\mathcal{U})g(\xi,\mathcal{U}).
\end{array}$$
\end{lemma}

\begin{theorem}\label{Th-Lie-ver}
  Let $(M,F,g)$ be a Finsler manifold and $G$ be a pseudo-Riemannian $g$-natural metric on $\widetilde{TM}$ of Kaluza Klein type. Given an arbitrary vector field $\xi $ on $M$, its vertical lift $\xi ^{v}$ is a conformal vector field on $(\widetilde{TM},G)$ if and only if the following conditions hold:
\begin{itemize}
\item [(i)] $\beta _{1}=0$ and there is a constant $\lambda \neq 0$ such that $\alpha _{1}+ \alpha_3=\lambda \alpha _{1}$;
\item [(ii)] $C(\xi ,.,.)=0$;
\item [(iii)] $g(\nabla _{X^h}\xi ,Y)+L(X,Y,\xi )=0$, for all $ X,Y\in \mathfrak{X}(M)$;
\item [(iv)] $2B(X,Y,\mathcal{U},\xi )-B(X,\xi ,Y,\mathcal{U})-B(Y,\xi ,X,\mathcal{U})=0$, for all $ X,Y\in \mathfrak{X}(M)$.
\end{itemize}
\end{theorem}

\begin{proof}
Since $G$ be of Kaluza-Klein type, then by virtue of Lemma \ref{Lem-Lie-ver} the Lie derivative of $G$ along $\xi^{v}$ is given by

$$\left\{
\begin{array}{rl}
\mathcal{L}_{\xi ^{v}}G(X^{h},Y^{h})=&\alpha _{1}(r{{}^2})(2B(X,Y,\mathcal{U},\xi )-B(X,\xi ,Y,\mathcal{U})-B(Y,\xi ,X,\mathcal{U}))
\\ 
&+2(\alpha _{1}+\alpha _{3})^{\prime }(r{{}^2})g(X,Y)g(\xi,\mathcal{U}) +2(\alpha _{1}+\alpha _{3})(r{{}^2})C(X,Y,\xi ) \\ \\
\mathcal{L}_{\xi ^{v}}G(X^{h},Y^{v})=&\alpha _{1}(r{{}^2})(g(\nabla _{X^h}\xi ,Y)+L(X,Y,\xi ))+\beta _{1}(r{{}^2})g(Y,\mathcal{U})g(\nabla _{X^h}\xi,\mathcal{U})  \\ \\
\mathcal{L}_{\xi ^{v}}G(X^{v},Y^{v})=&2\alpha _{1}^{\prime }(r{{}^2})g(X,Y)g(\xi,\mathcal{U})+2\alpha _{1}(r{{}^2})C(X,Y,\xi )\\
&+2\beta _{1}^{\prime }(r{{}^2})g(X,\mathcal{U})g(Y,\mathcal{U})g(\xi,\mathcal{U}) \\ 
&+\beta _{1}(r{{}^2})(g(X,\xi )g(Y,\mathcal{U})+g(Y,\xi )g(X,\mathcal{U}))%
\end{array}\right.
$$

If we suppose that $\xi ^{v}$ is conformal on $(TM,G)$, with potential function $\theta$, then the preceding system becomes

$$\left\{
\begin{array}{l}
2\theta (\alpha _{1}+\alpha _{3})(r^{2})g(X,Y)=\alpha _{1}(r{{}^2})(2B(X,Y,\mathcal{U},\xi )-B(X,\xi ,Y,\mathcal{U})-B(Y,\xi ,X,\mathcal{U}))
\\ 
\qquad\qquad\qquad\qquad+2(\alpha _{1}+\alpha _{3})^{\prime }(r{{}^2})g(X,Y)g(\xi,\mathcal{U})+2(\alpha _{1}+\alpha _{3})(r{{}^2})C(X,Y,\xi ) \\ \\
0=\alpha _{1}(r{{}^2})(g(\nabla _{X^h}\xi ,Y)+L(X,Y,\xi ))+\beta _{1}(r{{}^2})g(Y,\mathcal{U})g(\nabla _{X^h}\xi,\mathcal{U})\  \\ \\
2\theta [\alpha _{1}(r^{2})g(X,Y)+\beta_{1}(r^{2})g(X,\mathcal{U})g(Y,\mathcal{U})]=2\alpha _{1}^{\prime }(r{{}^2})g(X,Y)g(\xi,\mathcal{U})+2\alpha _{1}(r{{}^2})C(X,Y,\xi ) \\ 
\qquad\qquad+2\beta _{1}^{\prime }(r{{}^2})g(X,\mathcal{U})g(Y,\mathcal{U})g(\xi,\mathcal{U})+\beta _{1}(r{{}^2})(g(X,\xi )g(Y,\mathcal{U})+g(Y,\xi )g(X,\mathcal{U}))%
\end{array}\right.
$$

Fix $(x,u) \in \widetilde{TM}$. If we take$Y$ such that $Y_x=u$, then second equation of the preceding system yields $g(\nabla _{X^h_u}\xi,u)=0$ and, since $(x,u)$ is arbitrary, we have $g(\nabla _{X^h}\xi,\mathcal{U})=0$. Replacing again into the second equation and using the fact that $\alpha_1$ doesn't vanish, we get 
$$g(\nabla _{X^h}\xi ,Y)+L(X,Y,\xi )=0.$$
Now, by taking $X_x=Y_x=u$ in the first equation, we get 
\begin{equation}\label{Lie-ver1}
  \theta (x,u)=\frac{\phi_{1}^{\prime }}{\phi _{1}}(r{{}^2})g(\xi_x,u)=\frac{(\alpha _{1}+\alpha _{3})^{\prime }}{(\alpha
_{1}+\alpha _{3})}(r{{}^2})g(\xi_x,u).
\end{equation}

On the other hand, if we choose $X$ and $Y$ such that  $u=Y_x\perp X_x$, then from the third equation of the preceding system we get $\beta _{1}(r^2)=0$ and, since $(x,u)$ is arbitrary, we have $\beta_1 =0$ everywhere. We deduce then from \eqref{Lie-ver1} that 
$$\frac{\alpha_{1}^{\prime }}{\alpha _{1}}=\frac{(\alpha _{1}+\alpha _{3})^{\prime }}{(\alpha_{1}+\alpha _{3})},$$
which yields $\alpha _{1}+ \alpha_3=\lambda \alpha _{1}$, for some constant $\lambda$.

Now, taking $X=Y$ such that $X_x \perp u$ and taking into account that $\theta (x,u)=\frac{\alpha_{1}^{\prime }}{\alpha _{1}}(r{{}^2})g(\xi_x,u)$ and that $\alpha_1$ doesn't vanish, then the third equation of the preceding system yields $C(X_x,X_x,\xi_x)=0$. Since $(x,u)$ is arbitrary, we obtain $C(X,X,\xi)=0$ and, by symmetry of $C$, $C(X,Y,\xi)=0$, for all $X, Y \in \mathfrak{X}(M)$. Finally, taking into account that $\theta (x,u)=\frac{(\alpha _{1}+\alpha _{3})^{\prime }}{(\alpha_{1}+\alpha _{3})}(r{{}^2})g(\xi_x,u)$, the first equation of the preceding system gives 
$$2B(X,Y,\mathcal{U},\xi )-B(X,\xi ,Y,\mathcal{U})-B(Y,\xi ,X,\mathcal{U})=0,$$ 
for all $X,Y\in \mathfrak{X}(M)$.

The converse part  of the theorem is trivial.
\end{proof}

\begin{remark}
  In Theorem \ref{Th-Lie-ver}, the potential function $f$ of a conformal vector field $\xi^v$ is given by
  $$\theta=\frac{\alpha _{1}^{\prime }}{\alpha _{1}} \circ r{{}^2}.g(\xi,\mathcal{U}).$$
  We deduce that $\theta$ is constant if and only it vanishes identically if and only if either $\xi=0$ or $\alpha_1$ is constant. We have then the following corollary:
\end{remark}

\begin{corollary}
  Let $(M,F,g)$ be a Finsler manifold and $G$ be a pseudo-Riemannian $g$-natural metric on $\widetilde{TM}$ of Kaluza Klein type. Given an arbitrary vector field $\xi $ on $M$, the following statements are equivalent

\begin{enumerate}
\item $\xi ^{v}$ is a homothetic vector field on $(\widetilde{TM},G)$;
\item $\xi ^{v}$ is a Killing vector field on $(\widetilde{TM},G)$;
\item The following conditions hold:
\begin{itemize}
\item [(i)] $\beta _{1}=0$ and $\alpha _{1}$ and $\alpha _{1}+\alpha _{3}$ are non-zero constants,
\item [(ii)] $C(\xi ,.,.)=0$ and $g(\nabla _{X}\xi ,Y)+L(X,Y,\xi )=0$, for all $X,Y\in \mathfrak{X}(M)$,
\item [(iii)] $2B(X,Y,\mathcal{U},\xi )-B(X,\xi ,Y,\mathcal{U})-B(Y,\xi ,X,\mathcal{U})=0$, for all $X,Y\in \mathfrak{X}(M)$.
\end{itemize}
\end{enumerate}
\end{corollary}


\subsection{Complete lifts of vector fields}


\begin{lemma}\label{Lem-Lie-com}
Let $(M,F,g)$ be a Finsler manifold and $G$ be a pseudo-Riemannian $F$-natural metric on the slit tangent bundle $\widetilde{TM}$. Let $\xi$ be a vector field on $M$. The Lie derivative of $G$ along $\xi ^{c}$ is given by:
$$\arraycolsep1.5pt
\begin{array}{rl}
\mathcal{L}_{\xi ^{c}}G(X^{h},Y^{h})=& (\alpha _{1}+\alpha _{3})(r{{}^2})[g(\nabla _{X^{h}}\xi ,Y)+g(\nabla _{Y^{h}}\xi ,X)+2C(X,Y,\nabla_{\zeta}\xi)] \\ 
&+(\beta _{1}+\beta _{3})(r{{}^2})[g(\nabla _{X^{h}}\xi, \mathcal{U})g(Y,\mathcal{U})+g(\nabla _{Y^{h}}\xi ,\mathcal{U})g(X,\mathcal{U})\\
&+g(X,\nabla _{\zeta}\xi )g(Y,\mathcal{U})+g(Y,\nabla _{\zeta}\xi )g(X,\mathcal{U})] \\ 
&+\alpha _{2}\left( r{{}^2}\right) [\mathcal{R}(\mathcal{U},X,\xi ,Y)+\mathcal{R}(\mathcal{U},Y,\xi ,X)\\
&+2B(Y,\xi ,\mathcal{U},X)+2B(X,\xi ,\mathcal{U},Y)+2B(X,Y,\mathcal{U},\xi )] \\ 
&+\alpha _{2}(r{{}^2})[g(\nabla ^{2}\xi (\zeta,X^h),Y)+g(\nabla ^{2}\xi (\zeta,Y^h),X)+2L(X,Y,\nabla _{\zeta}\xi )] \\ 
&+\beta _{2}(r{{}^2})[g(X,\mathcal{U})g(\nabla ^{2}\xi (\zeta,Y^h),\mathcal{U})+g(Y,\mathcal{U})g(\nabla ^{2}\xi (\zeta,X^h),\mathcal{U})]\\ 
&+\alpha _{1}(r{{}^2})[2B(X,Y,\mathcal{U},\nabla _{\zeta}\xi )-B(X,\nabla _{\zeta}\xi ,Y,\mathcal{U})-B(Y,\nabla _{\zeta}\xi ,X,\mathcal{U})] \\ 
&+2[(\alpha _{1}+\alpha _{3})^{\prime}(r{{}^2})g(X,Y) +(\beta_{1}+\beta_{3})^{\prime}(r{{}^2})g(X,\mathcal{U})g(Y,\mathcal{U})]g(\nabla_{\zeta}\xi,\mathcal{U}) \\  
& \\
\mathcal{L}_{\xi ^{c}}G(X^{h},Y^{v})=&\alpha _{2}(r{{}^2})[g(\nabla _{X^{h}}\xi ,Y)+g(\nabla _{Y^{h}}\xi ,X)+2C(X,Y,\nabla _{\zeta}\xi )-L(X,Y,\xi )] \\ 
&+\beta _{2}(r{{}^2})[g(\nabla _{X^{h}}\xi, \mathcal{U})g(Y,\mathcal{U})+g(\nabla _{Y^{h}}\xi ,\mathcal{U})g(X,\mathcal{U})\\
&+g(X,\nabla _{\zeta}\xi )g(Y,\mathcal{U})+g(Y,\nabla _{\zeta}\xi )g(X,\mathcal{U})] \\ 
&+\alpha _{1}(r{{}^2})[\mathcal{R}(\mathcal{U},Y,\xi ,X)+B(Y,\xi ,\mathcal{U},X) +L(X,Y,\nabla _{\zeta}\xi )\\
&+g(\nabla ^{2}\xi (\zeta,X^h),Y)] +\beta _{1}(r{{}^2})g(Y,\mathcal{U})g(\nabla ^{2}\xi (\zeta,X^h),\mathcal{U})\\ 
&+2[\alpha _{2}^{\prime }(r{{}^2})g(X,Y)+\beta _{2}^{\prime }(r{{}^2})g(X,\mathcal{U})g(Y,\mathcal{U})]g(\nabla _{\zeta}\xi ,\mathcal{U})\\ 
&\\
\mathcal{L}_{\xi ^{c}}G(X^{v},Y^{v})=&\alpha _{1}(r{{}^2})[g(\nabla _{X^{h}}\xi ,Y)+g(\nabla _{Y^{h}}\xi ,X)+2C(X,Y,\nabla _{\zeta}\xi )-2L(X,Y,\xi )] \\ 
&+\beta _{1}(r{{}^2})[g(\nabla _{X^{h}}\xi, \mathcal{U})g(Y,\mathcal{U})+g(\nabla _{Y^{h}}\xi ,\mathcal{U})g(X,\mathcal{U})\\
&+g(X,\nabla _{\zeta}\xi )g(Y,\mathcal{U})+g(Y,\nabla _{\zeta}\xi )g(X,\mathcal{U})] \\ 
&+2[\alpha _{1}^{\prime }(r{{}^2})g(X,Y)+\beta _{1}^{\prime }(r{{}^2})g(X,\mathcal{U})g(Y,\mathcal{U})]g(\nabla _{\zeta}\xi ,\mathcal{U}).
\end{array}$$
\end{lemma}

It turns out that the investigation of general pseudo-Riemannian $F$-natural metrics of Kaluza-Klein type on the slit tangnet bundle of a Finslerian manifold is very hard. So we will restrict ourselves to some particular metrics, e.g. the Sasaki metric and the Cheeger-Gromoll metric. 

\begin{theorem}\label{Th-Lie-com}
  Let $(M,F,g)$ be a Finsler manifold and $G$ be the Sasaki metric on $\widetilde{TM}$. Given an arbitrary vector field $\xi $ on $M$, the following statements are equivalent:
\begin{enumerate}
\item $\xi ^{c}$ is a conformal vector field on $(\widetilde{TM},G)$;
\item $\xi ^{c}$ is a Killing vector field on $(\widetilde{TM},G)$;
\item The following identities hold:
\begin{itemize}
\item [(i)] $\xi$ is a Killing vector field on $(M,F,g)$,
\item [(ii)] $L(\xi ,.,.)=0$,
\item [(iii)] $2B(X,Y,\mathcal{U},(\nabla _{\mathcal{U}}\xi ))-B(X,(\nabla _{\mathcal{U}}\xi ),Y,\mathcal{U})-B(Y,(\nabla _{\mathcal{U}}\xi ),X,\mathcal{U})=0$,
\item [(iv)] $R(\mathcal{U},Y,\xi ,X)+B(Y,\xi ,\mathcal{U},X)+g(\nabla ^{2}\xi (\mathcal{U},X),Y)+L(X,Y,\nabla _{\mathcal{U}}\xi )=0$.
\end{itemize}
\end{enumerate}
\end{theorem}

\begin{proof}
Let $G$ be the Sasaki metric, and we suppse $\xi ^{c}$ is a conformal vector field of potential function $\theta$, then we have

$$\left\{ 
\begin{array}{l}
2\theta g(X,Y)=g(\nabla _{X^h}\xi ,Y)+g(\nabla _{Y^h}\xi ,X)+2C(X,Y,\nabla _{\zeta }\xi ) \\ 
\qquad\qquad\quad+2B(X,Y,\mathcal{U},\nabla _{\zeta }\xi )-B(X,\nabla _{\zeta }\xi ,Y,\mathcal{U})-B(Y,\nabla _{\zeta }\xi ,X,\mathcal{U}), \\ \\
0=\mathcal{R}(\mathcal{U},Y,\xi ,X)+B(Y,\xi ,\mathcal{U},X)+g(\nabla ^{2}\xi (\mathcal{U},X),Y)+L(X,Y,\nabla _{\zeta }\xi ), \\ \\
2\theta g(X,Y)=g(\nabla _{X^h}\xi ,Y)+g(\nabla _{Y^h}\xi,X)+2C(X,Y,\nabla _{\zeta }\xi )-2L(X,Y,\xi ).
\end{array}%
\right. $$
Combining the first and the third equations of the system, we get

$$2B(X,Y,\mathcal{U},\nabla _{\zeta }\xi )-B(X,\nabla _{\zeta }\xi ,Y,\mathcal{U})-B(Y,\nabla _{\zeta }\xi ,X,\mathcal{U})=-2L(X,Y,\xi )$$

Since the right and left sides of the last equation are homogeneous of degree $0$ and $2$, respectively then both sides are $0$. 
It follows that the system is rewritten in the form
$$\left\{
\begin{array}{l}
L(\xi ,.,.)=0,\\ \\
2B(X,Y,\mathcal{U},(\nabla _{\zeta }\xi ))-B(X,(\nabla _{\zeta }\xi ),Y,\mathcal{U})-B(Y,(\nabla _{\zeta }\xi ),X,\mathcal{U})=0, \\ \\
R(\mathcal{U},Y,\xi ,X)+B(Y,\xi ,\mathcal{U},X)+g(\nabla ^{2}\xi (\mathcal{U},X),Y)+L(X,Y,\nabla _{\zeta }\xi )=0, \\ \\
2\theta g(X,Y)=g(\nabla _{X^h}\xi ,Y)+g(\nabla _{Y^h}\xi ,X)+2C(X,Y,\nabla _{\zeta }\xi ).
\end{array}\right.
$$
Note that the last equation of the preceding system is equivalent to the fact that $\xi$ is a conformal vector field on $(M,F,g)$ with potential function $\theta$. On the other hand, Fixing $(x,u) \in \widetilde{TM}$ and taking $X$ and $Y$ such that $X_x=Y_x=u$, we get 
$\theta (x,u)=\frac{1}{r^{2}}g(\nabla _{\zeta }\xi ,u)$. Since $(x,u)$ is arbitrary and the right hand side of the last expression depends only on $(x,u)$, then we have
$$\theta=\frac{1}{r^{2}}g(\nabla _{\zeta }\xi ,\mathcal{U}).$$

The converse part of the theorem is straightforward.
\end{proof}

From Theorem \ref{Th-Lie-com}, there is no conformal or homothetic non-Killing complete lift vector field, when the slit tangent bundle is endowed with the Sasaki metric. This result is no longer true in the case of the Cheeger-Gromoll metric:

\begin{theorem}\label{Th-Lie-com2}
  Let $(M,F,g)$ be a Finsler manifold and $G$ be the Cheeger-Gromoll metric on $\widetilde{TM}$. Given an arbitrary vector field $\xi $ on $M$, then its complete lift  $\xi ^{c}$ to $\widetilde{TM}$ is a conformal vector field on $(\widetilde{TM},G)$ with potential function $\theta$ if and only if the following identities hold:
\begin{itemize}
\item [(i)] $\theta=\frac{1}{r^2}g(\nabla_\zeta \xi,\mathcal{U})$;
\item [(ii)] $\xi$ is a conformal vector field on $(M,F,g)$ with potential function $\theta$;
\item [(iii)] $2B(X,Y,\mathcal{U},(\nabla _{\zeta}\xi))-B(X,(\nabla _{\zeta}\xi ),Y,\mathcal{U}) -B(Y,(\nabla_{\zeta}\xi),X,\mathcal{U})=0$;
\item [(iv)] $\mathcal{R}(\mathcal{U},Y,\xi ,X)+B(Y,\xi ,\mathcal{U},X)+g(\nabla^{2}\xi (\zeta,X^h),Y)+L(X,Y,\nabla _{\zeta}\xi )=0$;
\item [(v)] $g(\nabla^{2}\xi (\zeta,X^h),\mathcal{U}) =0$;
\item [(vi)] $(g(\nabla _{X}\xi,\mathcal{U}) +g(X,\nabla _{\zeta}\xi))g(Y,\mathcal{U})+(g(\nabla _{Y}\xi,\mathcal{U})+g(Y,\nabla _{\zeta}\xi ))g(X,\mathcal{U})$\\
    $-\frac{4}{r^{2}}g(X,\mathcal{U})g(Y,\mathcal{U})g(\nabla _{\zeta}\xi,\mathcal{U})=0$;
\item [(vii)] $-2L(X,Y,\xi )-[\frac{2}{1+r^{2}}g(X,Y)-\frac{2}{r^{2}(1+r^{2})}g(X,\mathcal{U})g(Y,\mathcal{U})]g(\nabla _{\zeta}\xi,\mathcal{U})=0$.
\end{itemize}
\end{theorem}

\begin{proof}
Let $G$ be of the Cheeger-Gromoll metric and suppose that $\xi ^{c}$ is a conformal vector field with potential function $\theta$, then we have
$$\left\{ 
\begin{array}{l}
2\theta g(X,Y)=g(\nabla _{X^h}\xi ,Y)+g(\nabla _{Y^h}\xi ,X)+2C(X,Y,\nabla _{\zeta}\xi ) \\ 
\qquad\qquad+\frac{1}{1+r^{2}}(2B(X,Y,\mathcal{U},\nabla _{\mathcal{U}}\xi)-B(X,\nabla _{\mathcal{U}}\xi ,Y,\mathcal{U})-B(Y,\nabla _{\mathcal{U}%
}\xi ,X,\mathcal{U})) \\ \\
0=\mathcal{R}(\mathcal{U},Y,\xi ,X)+B(Y,\xi ,\mathcal{U},X)+g(\nabla ^{2}\xi (\mathcal{U},X),Y)+L(X,Y,\nabla _{\mathcal{U}}\xi )\\
\qquad+g(Y,\mathcal{U})g(\nabla ^{2}\xi (\mathcal{U},X),\mathcal{U}) \\ \\
2\theta [g(X,Y)+g(X,\mathcal{U})g(Y,\mathcal{U})]=g(\nabla _{X^h}\xi ,Y)+g(\nabla _{Y^h}\xi ,X)-2L(X,Y,\xi ) \\ 
\qquad\qquad\qquad\qquad+2C(X,Y,\nabla _{\zeta}\xi)+(g(\nabla _{X^h}\xi ,\mathcal{U})+g(X,\nabla _{\zeta}\xi ))g(Y,\mathcal{U})\\
\qquad\qquad\qquad\qquad +(g(\nabla_{Y^h}\xi ,\mathcal{U})+g(Y,\nabla _{\zeta}\xi ))g(X,\mathcal{U})\\
\qquad\qquad\qquad\qquad-\frac{2}{1+r^{2}}g(X,Y)g(\nabla_{\zeta}\xi ,\mathcal{U})-\frac{2}{1+r^{2}}g(X,\mathcal{U})g(Y,\mathcal{U})g(\nabla _{\zeta}\xi ,\mathcal{U})%
\end{array}%
\right. $$

By homogeneity arguments, as in the proof of Theorem \ref{Th-Lie-com}, the second equation of the preceding system yields 
$$\left\{\begin{array}{l}
    g(\nabla ^{2}\xi (\mathcal{U},X),\mathcal{U})=0, \\ \\
    \mathcal{R}(\mathcal{U},Y,\xi ,X)+B(Y,\xi,\mathcal{U},X)+g(\nabla ^{2}\xi (\mathcal{U},X),Y)+L(X,Y,\nabla _{\zeta}\xi )=0.
  \end{array}\right.
$$

Fix $(x,u) \in \widetilde{TM}$. Choosing $X$ and $Y$ such that $X_x=Y_x=u$, the first equation of the system yields $\theta (x,u)=\frac{1}{r{{}^2}} g(\nabla _{u^h}\xi ,u)$ and since $(x,u)$ is arbitrary, we get
$$\theta=\frac{1}{r{{}^2}} g(\nabla _{\zeta}\xi ,\mathcal{U}).$$

With the obtained conditions, we get  
$$\left\{ 
\begin{array}{l}
\frac{2}{r{{}^2}}g(\nabla _{\zeta}\xi,\mathcal{U})g(X,Y)=g(\nabla _{X^h}\xi ,Y)+g(\nabla _{Y^h}\xi,X)+2C(X,Y,\nabla _{\zeta}\xi ) \\ 
\qquad\qquad+\frac{1}{1+r^{2}}(2B(X,Y,\mathcal{U},\nabla _{\zeta}\xi)-B(X,\nabla _{\zeta}\xi ,Y,\mathcal{U})-B(Y,\nabla _{\zeta}\xi ,X,\mathcal{U})) \\ \\
\frac{2}{r^{2}}g(\nabla _{\zeta}\xi ,\mathcal{U})g(X,Y)=g(\nabla _{X^h}\xi,Y)+g(\nabla _{Y^h}\xi ,X)+2C(X,Y,\nabla _{\zeta}\xi )-2L(X,Y,\xi ) \\ 
\qquad\qquad-\frac{2}{1+r^{2}}g(X,Y)g(\nabla _{\zeta}\xi ,\mathcal{U})+(g(\nabla _{X^h}\xi,\mathcal{U})+g(X,\nabla _{\zeta}\xi ))g(Y,\mathcal{U})\\
\qquad\qquad+(g(\nabla _{Y^h}\xi ,\mathcal{U})+g(Y,\nabla _{\zeta}\xi ))g(X,\mathcal{U})-\frac{2(2r^{2}+1)}{r^{2}(1+r^{2})}g(X,\mathcal{U})g(Y,\mathcal{U})g(\nabla _{\zeta}\xi ,\mathcal{U})%
\end{array}%
\right. $$

Fixing $(x,u) \in \widetilde{TM}$ and taking the value of the first equation of the preceding system at $tu$, $t>0$, we get by arguments of homogeneity
$$\begin{array}{l}
\frac{2}{r{{}^2}}g(\nabla _{\zeta}\xi,\mathcal{U})g(X,Y)=g(\nabla _{X^h}\xi ,Y)+g(\nabla _{Y^h}\xi,X)+2C(X,Y,\nabla _{\zeta}\xi ) \\ 
\qquad\qquad+\frac{t^2}{1+t^2r^{2}}(2B(X,Y,\mathcal{U},\nabla _{\zeta}\xi)-B(X,\nabla _{\zeta}\xi ,Y,\mathcal{U})-B(Y,\nabla _{\zeta}\xi ,X,\mathcal{U})). \\ 
\end{array}$$
Comparing this equation with the first equation of the preceding system, we get 
$$\frac{1-t^2}{(1+r^2)(1+t^2r^{2})}(2B(X,Y,\mathcal{U},\nabla _{\zeta}\xi)-B(X,\nabla _{\zeta}\xi ,Y,\mathcal{U})-B(Y,\nabla _{\zeta}\xi, X,\mathcal{U}))=0,$$
for any $t>0$, and consequently
$$2B(X,Y,\mathcal{U},\nabla _{\zeta}\xi)-B(X,\nabla _{\zeta}\xi ,Y,\mathcal{U})-B(Y,\nabla _{\zeta}\xi, X,\mathcal{U})=0.$$
We deduce again from the first equation of the preceding system that
$$\frac{2}{r{{}^2}}g(\nabla _{\zeta}\xi,\mathcal{U})g(X,Y)=g(\nabla _{X^h}\xi ,Y)+g(\nabla _{Y^h}\xi,X)+2C(X,Y,\nabla _{\zeta}\xi ).$$
Taking into account the last equation, the second equation of the preceding system becomes
$$\begin{array}{l}
0=-2L(X,Y,\xi ) -\frac{2}{1+r^{2}}g(X,Y)g(\nabla _{\zeta}\xi ,\mathcal{U})+(g(\nabla _{X^h}\xi,\mathcal{U})+g(X,\nabla _{\zeta}\xi ))g(Y,\mathcal{U})\\
\qquad+(g(\nabla _{Y^h}\xi ,\mathcal{U})+g(Y,\nabla _{\zeta}\xi ))g(X,\mathcal{U})-\frac{2(2r^{2}+1)}{r^{2}(1+r^{2})}g(X,\mathcal{U})g(Y,\mathcal{U})g(\nabla _{\zeta}\xi ,\mathcal{U}).
\end{array}$$
Using the same technique as for the first equation of the system, we get

$$\left\{
\begin{array}{l}
0=(g(\nabla _{X^h}\xi,\mathcal{U})+g(X,\nabla _{\zeta}\xi ))g(Y,\mathcal{U}) +(g(\nabla _{Y^h}\xi ,\mathcal{U})+g(Y,\nabla _{\zeta}\xi ))g(X,\mathcal{U})\\
\qquad -\frac{4}{r^{2}}g(X,\mathcal{U})g(Y,\mathcal{U})g(\nabla _{\zeta}\xi ,\mathcal{U}), \\ \\
0=-2L(X,Y,\xi )-\frac{2}{1+r^{2}}g(X,Y)g(\nabla _{\zeta}\xi ,\mathcal{U})+\frac{2}{r^{2}(1+r^{2})}g(X,\mathcal{U})g(Y,\mathcal{U})g(\nabla _{\zeta}\xi ,\mathcal{U}).
\end{array}\right.
$$

The converse part of the theorem is straightforward. 
\end{proof}

\subsection{Vertical vector fields of the form $\iota P$}


For any $(1,1)$-tensor section $P$ on $\pi_0^*TM$, we can define a vertical vector field $\iota P$ on $\widetilde{TM}$, by 
$\iota P=v\{P(\mathcal{U})\}$.

\begin{lemma}\label{Lem-Lie-iota}
Let $(M,F,g)$ be a Finsler manifold and $G$ be a pseudo-Riemannian $F$-natural metric on the slit tangent bundle $\widetilde{TM}$. Let $P$ be a $(1,1)$-tensor section on $\pi_0^*TM$. The Lie derivative of $G$ along $\iota P$ is given by:
$$\arraycolsep1.5pt
\begin{array}{rl}
\mathcal{L}_{\iota P}G(X^{h},Y^{h})=& \alpha _{2}(r{{}^2})[g((\nabla _{X^{h}}P)(\mathcal{U}),Y)+g((\nabla _{Y^{h}}P)(\mathcal{U}),X)+2L(X,Y,P(\mathcal{U}))] \\ 
&+\beta _{2}(r{{}^2})[g(X,\mathcal{U})g(\nabla _{Y^{h}}P(\mathcal{U}),\mathcal{U})+g(Y,\mathcal{U})g(\nabla _{X^{h}}P(\mathcal{U}),\mathcal{U})]\\
&+2(\beta _{1}+\beta _{3})^{\prime }(r{{}^2})g(X,\mathcal{U})g(Y,\mathcal{U})g(P(\mathcal{U}),\mathcal{U}) \\ 
&+\alpha _{1}(r{{}^2})[2B(X,Y,\mathcal{U},P(\mathcal{U}))-B(X,P(\mathcal{U}),Y,\mathcal{U})-B(Y,P(\mathcal{U}),X,\mathcal{U})] \\ 
&+2(\alpha _{1}+\alpha _{3})^{\prime }(r{{}^2})g(X,Y)g(P(\mathcal{U}),\mathcal{U})+2(\alpha _{1}+\alpha _{3})(r{{}^2})C(X,Y,P(\mathcal{U})) \\ 
&+(\beta _{1}+\beta _{3})(r{{}^2})\left[ g(X,P(\mathcal{U}))g(Y,\mathcal{U})+g(Y,P(\mathcal{U}))g(X,\mathcal{U})\right] \\  
& \\
\mathcal{L}_{\iota P}G(X^{h},Y^{v})=&\alpha _{1}(r{{}^2})[g((\nabla _{X^{h}}P)(\mathcal{U}),Y)+L(X,Y,P(\mathcal{U}))]\\
&+\beta _{1}(r{{}^2})g(Y,\mathcal{U})g(\nabla _{X^{h}}P(\mathcal{U}),\mathcal{U}) +2\alpha _{2}^{\prime }(r{{}^2})g(X,Y)g(P(\mathcal{U}),\mathcal{U})  \\ 
&+\beta _{2}(r{{}^2})[g(X,\mathcal{U})\{g(Y,P(\mathcal{U}))+g(P(Y),\mathcal{U})\}+g(X,P(\mathcal{U}))g(Y,\mathcal{U})]
\\ 
&+\alpha _{2}(r{{}^2})[2C(X,Y,P(\mathcal{U}))+g(P(Y),X)]\\
&+2\beta _{2}^{\prime }(r{{}^2})g(X,\mathcal{U})g(Y,\mathcal{U})g(P(\mathcal{U}),\mathcal{U})\\
&\\
\mathcal{L}_{\iota P}G(X^{v},Y^{v})=&2\alpha _{1}^{\prime }(r{{}^2})g(X,Y)g(P(\mathcal{U}),\mathcal{U})+2\beta _{1}^{\prime }(r{{}^2})g(X,\mathcal{U})g(Y,\mathcal{U})g(P(\mathcal{U}),\mathcal{U}) \\ 
&+\alpha _{1}(r{{}^2})[2C(X,Y,P(\mathcal{U}))+g(P(X),Y)+g(P(Y),X)] \\ 
&+\beta _{1}(r{{}^2})[\{g(X,P(\mathcal{U}))+g(P(X),\mathcal{U})\}g(Y,\mathcal{U})\\
&+\{g(Y,P(\mathcal{U}))+g(P(Y),\mathcal{U})\}g(X,\mathcal{U})].
\end{array}$$
\end{lemma}

\begin{theorem}\label{Th-Lie-iota1}
  Let $(M,F,g)$ be a Finsler manifold and $G$ be the Sasaki metric or the Cheeger-Gromoll on $\widetilde{TM}$. Given an arbitrary  $(1,1)$-tensor section $P$ on $\pi_0^*TM$, then the following statements are equivalent
\begin{enumerate}
\item $\iota P$ is a conformal vector field on $(\widetilde{TM},G)$;
\item $\iota P$ is a Killing vector field on $(\widetilde{TM},G)$;
\item the following conditions hold
\begin{itemize}
\item [(i)] $P$ is skew-symmetric with respect to $g$, i.e. $g(P(X),Y)+g(P(Y),X)=0$, for all $X,Y \in \mathfrak{X}(M)$,
\item [(ii)] $2B(X,Y,\mathcal{U},P(\mathcal{U}))-B(X,P(\mathcal{U}),Y,\mathcal{U})-B(Y,P(\mathcal{U}),X,\mathcal{U})=0$,
\item [(iii)] $g((\nabla _{X^h}P)(\mathcal{U}),Y)+L(X,Y,P(\mathcal{U}))=0$,
\item [(iv)] $C(P(\mathcal{U}),.,.)=0$.
\end{itemize}
\end{enumerate}
\end{theorem}

\begin{proof}
The proof in the case when $G$ is the Sasaki metric is straightforward. We suppose that $G$ is the Cheeger-Gromoll metric and that $\iota P$ is a conformal vector field on $(\widetilde{TM},G)$. Then we have

$$\left\{ 
\begin{array}{l}
2\theta (x,\mathcal{U})g(X,Y)=\frac{1}{1+r^{2}}(2B(X,Y,\mathcal{U},P(\mathcal{U}))-B(X,P(\mathcal{U}),Y,\mathcal{U})\\ 
\qquad\qquad\qquad\qquad\qquad-B(Y,P(\mathcal{U}),X,%
\mathcal{U}))+2C(X,Y,P(\mathcal{U})), \\ \\
0=\frac{1}{1+r^{2}}[g((\nabla _{X^h}P)(\mathcal{U}),Y)+L(X,Y,P(\mathcal{U}))+g(Y,\mathcal{U})g((\nabla _{X^h}P)(\mathcal{U}),\mathcal{U})], \\ \\
2\theta (x,\mathcal{U})[g(X,Y)+g(X,\mathcal{U})g(Y,\mathcal{U})]=\frac{-2}{1+r^{2}}g(X,Y)g(P(\mathcal{U}),\mathcal{U}) \\ 
\qquad\qquad\qquad\qquad-\frac{2}{1+r^{2}}g(X,\mathcal{U})g(Y,\mathcal{U})g(P(\mathcal{U}),\mathcal{U})+2C(X,Y,P(\mathcal{U}))\\
\qquad\qquad\qquad\qquad+g(P(X),Y)+g(P(Y),X) \\ 
\qquad\qquad\qquad\qquad+\{g(X,P(\mathcal{U}))+g(P(X),\mathcal{U})\}g(Y,\mathcal{U})\\
\qquad\qquad\qquad\qquad+\{g(Y,P(\mathcal{U}))+g(P(Y),\mathcal{U})\}g(X,\mathcal{U}).
\end{array}%
\right. $$

Fixing $(x,u) \in \widetilde{TM}$. Choosing $X$ and $Y$ such that $X_x=Y_x=u$, then the first equation of the preceding system yields $\theta (x,u)=0$. Since $(x,u)$ is arbitrary, then $\theta$ is identically zero and hence $\iota P$ is a Killing vector field. Using arguments of homogeneity of the tensor sections $g$, $C$, $L$ and $B$, the preceding system is equivalent to
$$\left\{ 
\begin{array}{l}
C(P(\mathcal{U}),.,.)=0, \\ \\
2B(X,Y,\mathcal{U},P(\mathcal{U}))-B(X,P(\mathcal{U}),Y,\mathcal{U})-B(Y,P(\mathcal{U}),X,\mathcal{U}))=0, \\  \\
g((\nabla _{X^h}P)(\mathcal{U}),\mathcal{U})=0, \\ \\
g((\nabla _{X^h}P)(\mathcal{U}),Y)+L(X,Y,P(\mathcal{U}))=0, \\ \\
\frac{-2}{1+r^{2}}g(X,Y)g(P(\mathcal{U}),\mathcal{U}) -\frac{2}{1+r^{2}}g(X,\mathcal{U})g(Y,\mathcal{U})g(P(\mathcal{U}),\mathcal{U})\\
+g(P(X),Y)+g(P(Y),X) +\{g(X,P(\mathcal{U}))+g(P(X),\mathcal{U})\}g(Y,\mathcal{U})\\
+\{g(Y,P(\mathcal{U}))+g(P(Y),\mathcal{U})\}g(X,\mathcal{U})=0.
\end{array}%
\right. $$

Fixing $(x,u) \in \widetilde{TM}$ and taking the value of the last equation of the preceding system at $tu$, $t>0$, we get by arguments of homogeneity
$$ 
\begin{array}{l}
\frac{-2t^2}{1+t^2r^{2}}g(X,Y)g(P(\mathcal{U}),\mathcal{U}) -\frac{2t^4}{1+t^2r^{2}}g(X,\mathcal{U})g(Y,\mathcal{U})g(P(\mathcal{U}),\mathcal{U})\\
+g(P(X),Y)+g(P(Y),X) +t^2\{g(X,P(\mathcal{U}))+g(P(X),\mathcal{U})\}g(Y,\mathcal{U})\\
+t^2\{g(Y,P(\mathcal{U}))+g(P(Y),\mathcal{U})\}g(X,\mathcal{U})=0.
\end{array}%
 $$
Comparing the last equation with the third equation of the preceding system, we get
$$\frac{2t^2}{1+t^2r^{2}}g(X,\mathcal{U})g(Y,\mathcal{U})g(P(\mathcal{U}),\mathcal{U})+g(P(X),Y)+g(P(Y),X)=0. $$
Fixing again $(x,u) \in \widetilde{TM}$ and taking the value of the last equation of the preceding system at $tu$, $t>0$, we get by arguments of homogeneity
$$\frac{2t^6}{1+t^4r^{2}}g(X,\mathcal{U})g(Y,\mathcal{U})g(P(\mathcal{U}),\mathcal{U})+g(P(X),Y)+g(P(Y),X)=0. $$

Comparing the two last equations, we obtain
$$g(P(X),Y)+g(P(Y),X)=0.$$
  
The converse part of the theorem is straightforward.
\end{proof}


\subsection{The Liouville vector field}


As a corollary of the previous theorem, the Liouville vector field $\mathcal{V}$, which corresponds by $\iota$ to the identity $(1,1)$-tensor section $P$ on $\pi_0^*TM$, can not be conformal with respect to the Sasaki metric or the Cheeger-Gromoll metric, since the identity is not skew-symmetric with respect to $g$.  We prove that if we endow the slit tangent bundle with some pseudo-Riemannian $g$-natural metrics on $\widetilde{TM}$, the geodesic vector field on $\widetilde{TM}$ becomes conformal.

\begin{theorem}\label{Th-Lie-iota2}
  Let $(M,F,g)$ be a Finsler manifold and $G$ be a pseudo-Riemannian $g$-natural metric on $\widetilde{TM}$ of Kaluza Klein type. Then we have
\begin{enumerate}
\item The Liouville vector field $\mathcal{V}$ on $\widetilde{TM}$ is a conformal vector field on $(\widetilde{TM},G)$ if and only if there is a function $\lambda:]0,+\infty[ \rightarrow \mathbb{R}$ which is everywhere non zero, such that 
    \begin{equation}\label{Liouv1}
      \begin{array}{c}
        \alpha_1=a_1\lambda, \quad \alpha_2=a_2 \sqrt{t}\lambda, \quad \alpha_1 + \alpha_3= a_3 t\lambda, \\
        \beta_1=b_1\frac{\lambda}{t}, \quad \beta_2=b_2 \frac{\lambda}{\sqrt{t}}, \quad \beta_1 + \beta_3= b_3 \lambda, 
      \end{array}
    \end{equation}
    where $a_1a_3 -a_2^2 \neq 0$ and $(a_1+b_1)(a_3+b_3)-(a_2+b_2)^2 \neq 0$.\\
    Furthermore, the potential function is given by $\theta(u)= 1+\frac{r^2\lambda^\prime(r^2)}{\lambda(r^2)}$, for all $u \in \widetilde{TM}$.
\item The Liouville vector field $\mathcal{V}$ on $\widetilde{TM}$ is a homothetic vector field on $(\widetilde{TM},G)$ with constant potential $\theta_0$ if and only if there is a function $\lambda:]0,+\infty[ \rightarrow \mathbb{R}$ which is everywhere non zero, such that 
    \begin{equation}\label{Liouv2}
      \begin{array}{c}
        \alpha_1=a_1t^{\theta_0-1}, \quad \alpha_2=a_2 t^{\theta_0-\frac12}, \quad \alpha_1 + \alpha_3= a_3 t^{\theta_0}, \\
        \beta_1=b_1t^{\theta_0-2}, \quad \beta_2=b_2 t^{\theta_0-\frac32}, \quad \beta_1 + \beta_3= b_3 t^{\theta_0-1}, 
      \end{array}
    \end{equation}
    where $a_1a_3 -a_2^2 \neq 0$ and $(a_1+b_1)(a_3+b_3)-(a_2+b_2)^2 \neq 0$.
\item The Liouville vector field $\mathcal{V}$ on $\widetilde{TM}$ is a Killing vector field on $(\widetilde{TM},G)$ if and only if there is a function $\lambda:]0,+\infty[ \rightarrow \mathbb{R}$ which is everywhere non zero, such that 
    \begin{equation}\label{Liouv3}
      \begin{array}{c}
        \alpha_1=a_1t^{-1}, \quad \alpha_2=a_2 t^{-\frac12}, \quad \alpha_1 + \alpha_3= a_3 , \\
        \beta_1=b_1t^{-2}, \quad \beta_2=b_2 t^{-\frac32}, \quad \beta_1 + \beta_3= b_3 t^{-1}, 
      \end{array}
    \end{equation}
    where $a_1a_3 -a_2^2 \neq 0$ and $(a_1+b_1)(a_3+b_3)-(a_2+b_2)^2 \neq 0$.
\end{enumerate}
\end{theorem}

\begin{proof}
The Liouville vector field $\mathcal{V}$ corresponds by $\iota $ to the identity $(1,1)$-tensor section $P$ on $\pi _{0}^{\ast }TM$ and, in this case by Lemma \ref{Lem-Lie-iota}, $\mathcal{V}$ is a conformal vector field on $(\widetilde{TM},G)$ with potential function $\theta$ if and only if  
$$\left\{
\begin{array}{l}
\theta [(\alpha _{1}+\alpha _{3})(r^{2})g(X,Y)+(\beta_{1}+\beta _{3})(r^{2})g(X,\mathcal{U})g(Y,\mathcal{U})]=\\
\qquad\qquad= r^{2}(\beta _{1}+\beta_{3})^{\prime }(r{{}^2})g(X,\mathcal{U})g(Y,\mathcal{U}) +r^{2}(\alpha _{1}+\alpha _{3})^{\prime }(r{{}^2})g(X,Y)\\ 
\qquad\qquad \quad+(\beta _{1}+\beta _{3})(r{{}^2})g(X,\mathcal{U})g(Y,\mathcal{U}), \\ \\
2\theta [\alpha _{2}(r^{2})g(X,Y)+\beta_{2}(r^{2})g(X,\mathcal{U})g(Y,\mathcal{U})]=\\
\qquad\qquad= 2r^{2}\alpha _{2}^{\prime }(r{{}^2})g(X,Y)+3\beta _{2}(r{{}^2})g(X,\mathcal{U})g(Y,\mathcal{U}) \\ 
\qquad\qquad\quad+\alpha _{2}(r{{}^2})g(Y,X)+2r^{2}\beta _{2}^{\prime }(r{{}^2})g(X,\mathcal{U})g(Y,\mathcal{U}), \\ \\
\theta [\alpha _{1}(r^{2})g(X,Y)+\beta_{1}(r^{2})g(X,\mathcal{U})g(Y,\mathcal{U})]=\\
\qquad\qquad=r^{2}\alpha _{1}^{\prime }(r{{}^2})g(X,Y) +r^{2}\beta _{1}^{\prime }(r{{}^2})g(X,\mathcal{U})g(Y,\mathcal{U}) \\ 
\qquad\qquad\quad +\alpha _{1}(r{{}^2})g(X,Y)+2\beta _{1}(r{{}^2})g(X,\mathcal{U})g(Y,\mathcal{U})%
\end{array}\right.
$$

Fixing $(x,u) \in \widetilde{TM}$ and taking $X$ and $Y$ such that $X_x=Y_x\perp u$ (resp. $X_x = Y_x=u$), we get
\begin{equation}\label{Liouv4}
\left\{
\begin{array}{l}
\theta (x,u)(\alpha _{1}+\alpha _{3})(r^{2})=r^{2}(\alpha_{1}+\alpha _{3})^{\prime }(r{{}^2}), \\ 
2\theta (x,u)\alpha _{2}(r^{2})=2r^{2}\alpha _{2}^{\prime }(r{{}^2})+\alpha _{2}(r{{}^2}), \\ 
\theta (x,u)\alpha _{1}(r^{2})=r^{2}\alpha _{1}^{\prime }(r{{}^2})+\alpha _{1}(r{{}^2}), \\ 
\theta (x,u)(\phi _{1}+\phi _{3})(r^{2})=r^{2}(\phi_{1}+\phi _{3})^{\prime }(r{{}^2}), \\ 
2\theta (x,u)\phi _{2}(r^{2})=2r^{2}\phi _{2}^{\prime }(r{{}^2})+\phi _{2}(r{{}^2}), \\ 
\theta (x,u)\phi _{1}(r^{2})=r^{2}\phi _{1}^{\prime }(r{{}^2})+\phi _{1}(r{{}^2}).
\end{array}\right.
\end{equation}

Since $(x,u)$ is arbitrary then the system holds for any $(x,u) \in \widetilde{TM}$ and $\mathcal{V}$ is a conformal vector field on $(\widetilde{TM},G)$ with potential function $\theta$ if and only if the system \eqref{Liouv4} holds.

It follows from \eqref{Liouv4} that $\theta$ is spherically symmetric in the sense that it depends only on the norm of vectors of $\widetilde{TM}$. From the regularity of $G$, $\alpha_1$ and $\alpha_2$ don't vanish simultaneously. Put $I_i=\{t\in ]0,+\infty[, \alpha_i(t) \neq 0\}$, $i=1,2$. Each $I_i$ is an open subset of $\mathbb{R}^+_*$, only one of them can be empty. 

From the second and third equations of \eqref{Liouv4}, we have
\begin{equation}\label{Liouv5}
  \theta(x,u)=\left\{\begin{array}{l}
                       1+ r^2 \frac{\alpha_1^\prime}{\alpha_1}(r^2), \quad \textup{whenever} \quad r^2 \in I_1, \\
                       \frac12  + r^2 \frac{\alpha_2^\prime}{\alpha_2}(r^2), \quad \textup{whenever} \quad r^2 \in I_2.
                     \end{array}\right.
\end{equation}
We claim that either $I_1=\mathbb{R}^+_*$ or $I_1=\emptyset$. We have to discuss two cases:
\begin{itemize}
  \item $I_1 \cap I_2 \neq \emptyset$. Then we have from \eqref{Liouv5}
  $$\frac1{2t}+ \frac{\alpha_1^\prime}{\alpha_1}(t)= \frac{\alpha_2^\prime}{\alpha_2}(t), \quad \textup{for each} \quad t \in I_1 \cap I_2,$$
  whose integration gives 
  \begin{equation}\label{Liouv6}
    \alpha_2(t) =a\sqrt{t}\alpha_1(t), \quad \textup{for each} \quad t \in I_1 \cap I_2,
  \end{equation}
  where $a$ is a constant. Since $I_1 \cap I_2$ is an open proper subset of $\mathbb{R}^+_*$, then its frontier is not empty. Let $t_0$ be in the frontier so that $t_0 \notin I_1 \cap I_2$ and there is $(t_n)_{n \in \mathbb{N}^*} \subset I_1 \cap I_2$ such that $\lim_{n\to\infty} t_n=t_0$. Then $\alpha_1(t_0)=0$ or $\alpha_2(t_0)=0$ and, by \eqref{Liouv6}, $\alpha_1(t_0)=\alpha_2(t_0)=0$ which contradicts the regularity of $G$.
  \item $I_1 \cap I_2 = \emptyset$. Remark that we have always $I_1 \cup I_2 = \mathbb{R}^+_*$. Indeed, suppose there is $t \in I_1 \cup I_2$, then $\alpha_1(t)=\alpha_2(t)=0$ which contradicts the regularity of $G$. We deduce that $\mathbb{R}^+_*$ is a disjoint union of two open sets. We deduce that either $I_1=\mathbb{R}^+_*$ or $I_1=\emptyset$ since  $\mathbb{R}^+_*$ is connected.
\end{itemize}
So, either $\alpha_1$ doesn't vanish and $\alpha_2=a\sqrt{t} \alpha_1$ on $\mathbb{R}^+_*$, or $\alpha_1$ vanishes identically on $\mathbb{R}^+_*$ and $\alpha_2$ doesn't vanish on $\mathbb{R}^+_*$. In both cases, either $I_2=\mathbb{R}^+_*$ or $I_2=\emptyset$.

Suppose that $\alpha_1$ doesn't vanish and $\alpha_2=a\sqrt{t} \alpha_1$ on $\mathbb{R}^+_*$ and let $I_3=\{t\in ]0,+\infty[, (\alpha_1+\alpha_3)(t) \neq 0\}$. Then from the first equation of the system \eqref{Liouv4}, either $I_3=\emptyset$, i.e. $\alpha_1+\alpha_3$ vanishes identically on $\mathbb{R}^+_*$;
or $I_3\neq \emptyset$ and, in this case, we claim that $I_3=\mathbb{R}^+_*$. Indeed, otherwise there is $t_0$ the frontier of the open set $I_3$. In other words, $t_0 \notin I_3$ and there is $(t_n)_{n \in \mathbb{N}^*} \subset I_3$ such that $\lim_{n\to\infty} t_n=t_0$.  Combining the first and third equation of \eqref{Liouv4} and solving the obtained differential equation on $I_3$, we get $\alpha_1+\alpha_3=bt\alpha_1$, where $b$ is a constant. Applying the last equation to $t_n$ and making the limit when $n\to\infty$, we get $0=(\alpha_1+\alpha_3)(t_0)=bt_0\alpha_1(t_0) \neq 0$, which is a contradiction. We deduce that either $I_3=\emptyset$ or $I_3=\mathbb{R}^+_*$. In all the cases, we have $\alpha_1+\alpha_3=bt\alpha_1$ on $\mathbb{R}^+_*$, where $b$ is constant.

Suppose that $\alpha_1$ vanishes identically. Then $\alpha_2$ doesn't vanish on $\mathbb{R}^+_*$. From the second equation of \eqref{Liouv4}, we have $\theta =\frac12  + r^2 \frac{\alpha_2^\prime}{\alpha_2}(r^2)$ on $\mathbb{R}^+_*$. the same arguments as before show that $\alpha_1+\alpha_3$ either vanishes identically or doesn't vanish on $\mathbb{R}^+_*$ and $\alpha_1+\alpha_3=c\sqrt{t}\alpha_2$.

Summing up the preceding discussion, we have
\begin{itemize}
  \item either $\alpha_1$ doesn't vanish on $\mathbb{R}^+_*$, $\alpha_2=a\sqrt{t}\alpha_1$ and $\alpha_1+\alpha_3=b t\alpha_1$,
  \item or $\alpha_1$ vanishes identically on $\mathbb{R}^+_*$ and $\alpha_1+\alpha_3=c \sqrt{t}\alpha_2$.
\end{itemize}

Using the same arguments as for the three first equations of \eqref{Liouv4} for the three last equations, we find that 
\begin{itemize}
  \item either $\phi_1$ doesn't vanish on $\mathbb{R}^+_*$, $\phi_2=a^\prime\sqrt{t}\phi_1$ and $\phi_1+\phi_3=b^\prime t\phi_1$,
  \item or $\phi_1$ vanishes identically on $\mathbb{R}^+_*$ and $\phi_1+\phi_3=c^\prime \sqrt{t}\phi_2$.
\end{itemize}
 
Taking into account all the cases discussed, we have the four following situations:
\begin{itemize}
  \item [Case 1:] $\alpha_1$ and $\phi_1$ doesn't vanish identically on $\mathbb{R}^+_*$, $\alpha_2=a\sqrt{t}\alpha_1$, $\alpha_1+\alpha_3=b t\alpha_1$, $\phi_2=a^\prime\sqrt{t}\phi_1$ and $\phi_1+\phi_3=b^\prime t\phi_1$. Then the same arguments as before applied to the third and sixth equations of \eqref{Liouv4} yields that $\phi_1=d \alpha_1$ on $\mathbb{R}^+_*$, where $d$ is a constant. We deduce that
      \begin{equation*}
        \left\{
        \begin{array}{l}
          t\beta_1= \phi_1 -\alpha_1 = (d-1)\alpha_1;\\ \\
          t\beta_2=  \phi_2 -\alpha_2 =\sqrt{t}(a^\prime \phi_1 -a\alpha_1) =\sqrt{t}(a^\prime d -a)\alpha_1; \\ \\
          t(\beta_1 +\beta_3) = \phi_1+\phi_3 -(\alpha_1+\alpha_3)=(b^\prime d-b)t\alpha_1.
        \end{array}
        \right.
      \end{equation*} 
  \item [Case 2:] $\alpha_1$ doesn't vanish and $\phi_1$ vanishes identically on $\mathbb{R}^+_*$, $\alpha_2=a\sqrt{t}\alpha_1$, $\alpha_1+\alpha_3=b t\alpha_1$ and $\phi_1+\phi_3=c^\prime \sqrt{t}\phi_2$. Then the same arguments as before applied to the second, fifth and sixth equations of \eqref{Liouv4} yields that $\phi_1=d_1 \alpha_1$ and $\phi_2=d_2 \sqrt{t}\alpha_1$ on $\mathbb{R}^+_*$, where $d_1$ and $d_2$ are constant. We deduce that
      \begin{equation*}
        \left\{
        \begin{array}{l}
          t\beta_1= \phi_1 -\alpha_1 =(d_1 -1) \alpha_1;\\ \\
          t\beta_2=  \phi_2 -\alpha_2 =(a- d_2)\sqrt{t}\alpha_1; \\ \\
          t(\beta_1 +\beta_3) = \phi_1+\phi_3 -(\alpha_1+\alpha_3)=c^\prime \sqrt{t}\phi_2-bt\alpha_1=(c^\prime d_2-b)t\alpha_1.
        \end{array}
        \right.
      \end{equation*}
  \item [Case 3:] $\alpha_1$ vanishes and $\phi_1$ doesn't vanish identically on $\mathbb{R}^+_*$, $\alpha_1+\alpha_3=c \sqrt{t}\alpha_2$, $\phi_2=a^\prime\sqrt{t}\phi_1$ and $\phi_1+\phi_3=b^\prime t\phi_1$. Then the same arguments as before applied to the second, fifth and sixth equations of \eqref{Liouv4} yields that $\phi_1=d_1^\prime \frac{1}{\sqrt{t}}\alpha_2$ and $\phi_2=d_2^\prime \alpha_2$ on $\mathbb{R}^+_*$, where $d_1^\prime$ and $d_2^\prime$ are constant. We deduce that
      \begin{equation*}
        \left\{
        \begin{array}{l}
          t\beta_1= \phi_1 -\alpha_1 = \phi_1=d_1^\prime \frac{1}{\sqrt{t}}\alpha_2;\\ \\
          t\beta_2=  \phi_2 -\alpha_2 =(d_2^\prime -1)\alpha_2 \\ \\
          t(\beta_1 +\beta_3) = \phi_1+\phi_3 -(\alpha_1+\alpha_3)=b^\prime t\phi_1-c \sqrt{t}\alpha_2=(b^\prime d_1^\prime-c)\sqrt{t}\alpha_2.
        \end{array}
        \right.
      \end{equation*}
  \item [Case 4:] $\alpha_1$ and $\phi_1$ vanish identically on $\mathbb{R}^+_*$, $\alpha_1+\alpha_3=c \sqrt{t}\alpha_2$  and $\phi_1+\phi_3=c^\prime \sqrt{t}\phi_2$. Then the same arguments as before applied to the second and fifth equations of \eqref{Liouv4} yields that $\phi_2=d^\prime \alpha_2$ on $\mathbb{R}^+_*$, where $d^\prime$ is a constant. We deduce that
      \begin{equation*}
        \left\{
        \begin{array}{l}
          t\beta_1= \phi_1 -\alpha_1 = 0;\\ \\
          t\beta_2=  \phi_2 -\alpha_2 =(d^\prime -1)\alpha_2 \\ \\
          t(\beta_1 +\beta_3) = \phi_1+\phi_3 -(\alpha_1+\alpha_3)=(c^\prime -c)\sqrt{t}\alpha_2.
        \end{array}
        \right.
      \end{equation*}
\end{itemize}
To complete the proof of the first assertion of the theorem, it suffices to take $\lambda=\alpha_1$ in the two first cases and $\lambda=\frac1{\sqrt{t}}\alpha_2$ in the last two cases. In all cases, the potential function is given by
\begin{equation}\label{pot_Liouv}
  \theta(u)= 1+\frac{r^2\lambda^\prime(r^2)}{\lambda(r^2)} \quad \textup{for all} u \in \widetilde{TM}.
\end{equation}

Homothetic (resp. Killing) Liouville vector field corresponds to $\theta$ is a constant $\theta_0$ (resp. 0) which gives, by solving the differential equation $1+\frac{t\lambda^\prime}{\lambda}=\theta_0$, $\lambda=\kappa t^{\theta_0-1}$ (resp. $\lambda=\kappa t^{-1}$), where $\kappa$ is a constant. 
\end{proof}


\subsection{The Geodesic vector field}


Like the vertical transvection $\iota$, we can define a horizontal transvection $\tau$ on $(1,1)$-tensor sections by the following: Let $P$ be a $(1,1)$-tensor section on $\pi_0^*TM$. $\tau P$ is the horizontal vector field on $\widetilde{TM}$ defined by
$\tau P=h\{P(\mathcal{U})\}$.
 
\begin{lemma}\label{Lem-Lie-tau}
Let $(M,F,g)$ be a Finsler manifold and $G$ a pseudo-Riemannian $F$-natural metric on $\widetilde{TM}.$ Let $P$ be a $(1,1)-$tensor section on $\pi
_{0}^{\ast }TM.$ The Lie derivative of $G$ along $\tau P$ is given by
$$%
\begin{array}{rl}
\mathcal{L}_{\tau P}G(X^{h},Y^{h})= & (\alpha _{1}+\alpha _{3})(r{{}^{2}})[g((\nabla _{X^{h}}P)(\mathcal{U}),Y)+g((\nabla _{Y^{h}}P)(\mathcal{U}),X)] \\ 
&  +(\beta _{1}+\beta _{3})(r{{}^{2}})[g(\nabla _{X^{h}}P)(\mathcal{U}),\mathcal{U})g(Y,\mathcal{U}) \\ 
& +g((\nabla _{Y^{h}}P)(\mathcal{U}),\mathcal{U})g(X,\mathcal{U})] \\ 
& -\alpha _{2}\left( r{{}^{2}}\right) [\mathcal{R}(X,\mathcal{U},P(\mathcal{U}),Y)+\mathcal{R}(Y,\mathcal{U},P(\mathcal{U}),X) \\ 
&+2C(P(\mathcal{U}),\mathcal{R}(\mathcal{U},X)\mathcal{U},Y)+2C(P(\mathcal{U}),\mathcal{R}(\mathcal{U},Y)\mathcal{U},X) \\ 
&+2C(X,\mathcal{R}(\mathcal{U},P(\mathcal{U}))\mathcal{U},Y)] \\ 
&  \\ 
\mathcal{L}_{\tau P}G(X^{h},Y^{v})= & \alpha _{2}(r{{}^{2}})[g((\nabla_{X^{h}}P)(\mathcal{U}),Y)-L(X,Y,P(\mathcal{U}))] \\ 
& +\beta _{2}(r{{}^{2}})g((\nabla _{X^{h}}P)(\mathcal{U}),\mathcal{U})g(Y,\mathcal{U}) \\ 
& +\alpha _{1}(r{{}^{2}})[\mathcal{R}(\mathcal{U},Y,P(\mathcal{U}),X)+B(Y,P(\mathcal{U}),\mathcal{U},X)] \\ 
& +(\alpha _{1}+\alpha _{3})(r^{2})g(X,P(Y)) \\ 
& +(\beta _{1}+\beta _{3})(r^{2})g(P(Y),\mathcal{U})g(P(X,\mathcal{U}) \\ 
&  \\ 
\mathcal{L}_{\tau P}G(X^{v},Y^{v})= & \alpha_{2}(r^{2})[g(P(X),Y)+g(X,P(Y))]-2\alpha _{1}(r{{}^2})L(X,Y,P(\mathcal{U})) \\ 
& +\beta _{2}(r^{2})[g(P(X),\mathcal{U})g(P(Y,\mathcal{U})+g(P(Y),\mathcal{U})g(P(X),\mathcal{U})] 
\end{array}%
$$
\end{lemma}

\begin{theorem}
Let $(M,F,g)$ be a Finsler manifold, $G$ be a Kaluza-Klein metric on $TM$ and $P$ be a non-zero tensor section on $\pi_0^*TM$. Then $\tau P$ can not be a
conformal vector field.
\end{theorem}

\begin{proof}
The proof can easily be achieved from Lemma \ref{Lem-Lie-tau} by supposing that $\tau P$ a conformal vector field and substituting $\alpha _{2}$ and $\alpha _{2}$ and $(\beta _{1}+\beta _{3})$ by $0$ (Kaluza-Klein metric) to get
$$\left\{
\begin{array}{l}
2\theta g(X,Y) =[g((\nabla _{X^{h}}P)(\mathcal{U}),Y) +g((\nabla_{Y^{h}}P)(\mathcal{U}),X)], \\ \\
0  =\alpha _{1}(r{{}^{2}})[\mathcal{R}(\mathcal{U},Y,P(\mathcal{U} ),X)+B(Y,P(\mathcal{U}),\mathcal{U},X)]+(\alpha _{1}+\alpha_{3})(r^{2})g(X,P(Y)), \\ 
\\ 
\theta [\alpha _{1}(r{{}^{2}})g(X,Y)+\beta _{1}(r^{2})g(X,\mathcal{U})g(Y,\mathcal{U})]  =-\alpha _{1}(r{{}^2})L(X,Y,P(\mathcal{U})). \\ 
\end{array}\right.
$$
Fixing $(x,u) \in \widetilde{TM}$ and taking $Y$ such that $Y_x=u$, the second equation of the preceding system yields $(\alpha _{1}+\alpha_{3})(r^{2})g(X_x,P(u))=0$, i.e. $g(X_x,P(u))=0$ (since $\alpha_1+\alpha_3$ dosn't vanish). Since $(x,u)$ and $X$ are arbitrary, then $P(u)=0$, for all $(x,u) \in \widetilde{TM}$, which is a contradiction.
\end{proof}

The geodesic vector field $\zeta$ on $\widetilde{TM}$ corresponds to $\tau I$, where $I$ is the identity $(1,1)-$tensor section on $\pi_{0}^{\ast }TM$. 
We have the following:

\begin{theorem}
Let $(M,F,g)$ be a Finsler manifold and $G$ be a pseudo-Riemannian $F$-natural metric on $\widetilde{TM}$. Then the geodesic vector field $\zeta$ can not be a conformal vector field on $(\widetilde{TM},G)$.
\end{theorem}

\begin{proof}
Using Lemma \ref{Lem-Lie-tau}, it is easy to see that the Lie derivative of $G$ with respect to $\zeta$ is characterized by the three identities
$$
\begin{array}{rl}
\mathcal{L}_{u^{h}}G(X^{h},Y^{h})= & -2\alpha _{2}\left( r{{}^{2}}\right)\mathcal{R}(X,\mathcal{U},\mathcal{U},Y), \\ 
&  \\ 
\mathcal{L}_{u^{h}}G(X^{h},Y^{v})= & (\alpha _{1}+\alpha_{3})(r^{2})g(X,Y)+(\beta _{1}+\beta _{3})(r^{2})g(X,\mathcal{U})g(Y,\mathcal{U}), \\ 
& +\alpha _{1}(r{{}^{2}})\mathcal{R}(\mathcal{U},Y,\mathcal{U},X) \\ 
&  \\ 
\mathcal{L}_{u^{h}}G(X^{v},Y^{v})= & 2\alpha _{2}(r^{2})g(X,Y)+2\beta_{2}(r^{2})g(X,\mathcal{U})g(Y,\mathcal{U}).
\end{array}%
$$
We suppose that $\zeta$ is conformal with potential fynction $\theta$, then we have by the preceding system
$$\left\{
\begin{array}{l}
\theta [(\alpha _{1}+\alpha _{3})(r^{2})g(X,Y)+(\beta _{1}+\beta_{3})(r^{2})g(X,\mathcal{U})g(Y,\mathcal{U})]  =-\alpha _{2}\left( r{{}^{2}}%
\right) \mathcal{R}(X,\mathcal{U},\mathcal{U},Y), \\ \\
2\theta [\alpha _{2}(r^{2})g(X,Y)+\beta _{2}(r^{2})g(X,\mathcal{U})g(Y,\mathcal{U})]  =\\
\qquad=\alpha _{1}(r{{}^{2}})\mathcal{R}(\mathcal{U},Y,\mathcal{U},X)  
+(\alpha _{1}+\alpha _{3})(r^{2})g(X,Y) +(\beta _{1}+\beta _{3})(r^{2})g(X,\mathcal{U})g(Y,\mathcal{U}),\\ \\
\theta [\alpha _{1}(r^{2})g(X,Y)+\beta _{1}(r^{2})g(X,\mathcal{U})g(Y,\mathcal{U})]  =\alpha _{2}(r^{2})g(X,Y)+\beta _{2}(r^{2})g(X,\mathcal{U}%
)g(Y,\mathcal{U}). 
\end{array}\right.
$$

Fixing $(x,u) \in \widetilde{TM}$ and taking $X$ and $Y$ such that $X_x=Y_x=u$, we get 
$$\left\{
\begin{array}{ll}
\theta (x,u)(\phi _{1}+\phi _{3})(r^{2}) & =0, \\ 
2\theta (x,u)\phi _{2}(r^{2}) & =(\phi _{1}+\phi _{3})(r^{2}), \\ 
\theta (x,u)\phi _{1}(r^{2}) & =\phi _{2}(r^{2}),
\end{array}\right.
$$
which contradicts the fact that $\phi=\phi_1(\phi _{1}+\phi _{3})-\phi _{2}^2$ doesn't vanish.
\end{proof}

\newpage
\begin{appendix}


\section{Proofs of technical results}


\emph{Proof of Proposition \ref{lev-civ-con}:} 
We calculate $\bar{\nabla} _{X^{h}}Y^{h}$, calculations for the other quantities being similar. Using Koszul formula, we have by virtue of Lemmas \ref{lem1}-\ref{lem4}, 
\begin{eqnarray*}
2G(\bar{\nabla} _{X^{h}}Y^{h},Z^{h})
&=&X^{h}(G(Y^{h},Z^{h}))+Y^{h}(G(Z^{h},X^{h}))-Z^{h}(G(X^{h},Y^{h})) \\
&&+G\left( \left[ X^{h},Y^{h}\right] ,Z^{h}\right) -G\left( \left[Y^{h},Z^{h}\right] ,X^{h}\right) +G\left( \left[ Z^{h},X^{h}\right],Y^{h}\right)\\
&=&(\alpha_1 +\alpha_3)(r^2)[g(\nabla_{X^h}Y,Z) +g(Y,\nabla_{X^h}Z) +g(\nabla_{Y^h}X,Z)\\
&& +g(X,\nabla_{Y^h}Z) -g(\nabla_{Z^h}X,Y)- g(X,\nabla_{Z^h}Y) \\
&&+g(\nabla_{X^h}Y,Z) -g(\nabla_{Y^h}X,Z) -g(\nabla_{Y^h}Z,X) \\
&&+g(\nabla_{Z^h}Y,X) +g(\nabla_{Z^h}X,Y) -g(\nabla_{X^h}Z,Y)] \\
&& +(\beta_1+\beta_3)(r^2)[g(\nabla_{X^h}Y,\mathcal{U})g(Z,\mathcal{U}) +g(Y,\mathcal{U})g(\nabla_{X^h}Z,\mathcal{U})\\
&& +g(\nabla_{Y^h}X,\mathcal{U})g(Z,\mathcal{U}) +g(X,\mathcal{U})g(\nabla_{Y^h}Z,\mathcal{U}) \\
&&-g(\nabla_{Z^h}Y,\mathcal{U})g(X,\mathcal{U}) - g(X,\mathcal{U})g(\nabla_{Z^h}Y,\mathcal{U})\\
&& +g(\nabla_{X^h}Y,\mathcal{U})g(Z,\mathcal{U}) -g(\nabla_{Y^h}X,\mathcal{U})g(Z,\mathcal{U})\\
&& -g(\nabla_{Y^h}Z,\mathcal{U})g(X,\mathcal{U}) +g(\nabla_{Z^h}Y,\mathcal{U})g(X,\mathcal{U})\\
&& +g(\nabla_{Z^h}X,\mathcal{U})g(Y,\mathcal{U}) -g(\nabla_{X^h}Z,\mathcal{U})g(Y,\mathcal{U})] \\
&& +\alpha_2(r^2)[-g(\mathcal{R}(X,Y)\mathcal{U},Z) +g(\mathcal{R}(Y,Z)\mathcal{U},X) -g(\mathcal{R}(Z,X)\mathcal{U},Y)]\\
&=& 2(\alpha_1 +\alpha_3)(r^2)g(\nabla_{X^h}Y,Z) +2(\beta_1+\beta_3)(r^2)g(\nabla_{X^h}Y,\mathcal{U})g(Z,\mathcal{U})\\
&& +\alpha_2(r^2)[-g(\mathcal{R}(X,Y)\mathcal{U},Z) +g(\mathcal{R}(\mathcal{U},X)Y,Z)\\
&&  -C(Y,Z,\mathcal{R}(\mathcal{U},X)\mathcal{U}) +C(X,Y,\mathcal{R}(\mathcal{U},Z)\mathcal{U})\\
&&   + C(Z,X,\mathcal{R}(Y,\mathcal{U})\mathcal{U}) +g(\mathcal{R}(\mathcal{U},Y)X,Z)\\
&& -C(X,Z,\mathcal{R}(\mathcal{U},Y)\mathcal{U}) +C(Y,X,\mathcal{R}(\mathcal{U},Z)\mathcal{U})\\
&&   + C(Z,Y,\mathcal{R}(X,\mathcal{U})\mathcal{U}) ]\\
&=& 2(\alpha_1 +\alpha_3)(r^2)g(\nabla_{X^h}Y,Z) +2(\beta_1+\beta_3)(r^2)g(\nabla_{X^h}Y,\mathcal{U})g(Z,\mathcal{U})\\
&& -2\alpha_2(r^2)[g(\mathcal{R}(X,\mathcal{U})Y,Z) -g(\bar{B}(X,Y),Z) \\
&& -g(\bar{C}(Y,\mathcal{R}(X,\mathcal{U})\mathcal{U}),Z)- g(\bar{C}(X,\mathcal{R}(Y,\mathcal{U})\mathcal{U}),Z)].
\end{eqnarray*}
We deduce that $G(\bar{\nabla} _{X^{h}}Y^{h},Z^{h})=g(P_1(X,Y),Z)$, where $P_1$ is the $(1,2)$-tensor section on $\pi_0^*TM$ defined by
\begin{equation*}
  \begin{split}
    P_1(\sigma_1,\sigma_2)= & (\alpha_1 +\alpha_3)(r^2)\nabla_{h\sigma_1}\sigma_2 +(\beta_1+\beta_3)(r^2)g(\nabla_{h\sigma_1}\sigma_2,\mathcal{U})\mathcal{U} \\
      &  -\alpha_2(r^2)[\mathcal{R}(\sigma_1,\mathcal{U})\sigma_2-\bar{B}(\sigma_1,\sigma_2)  \\ 
      & -\bar{C}(\sigma_2,\mathcal{R}(\sigma_1,\mathcal{U})\mathcal{U})-\bar{C}(\sigma_1,\mathcal{R}(\sigma_2,\mathcal{U})\mathcal{U})].
  \end{split}
\end{equation*}

On the other hand, using the same arguments as before, we get $G(\bar{\nabla} _{X^{h}}Y^{h},Z^{v})=g(P_2(X,Y),Z)$, where $P_2$ is the $(1,2)$-tensor section on $\pi_0^*TM$ defined by
\begin{equation*}
  \begin{split}
    P_2(\sigma_1,\sigma_2)= & \alpha_2(r^2)[\nabla_{h\sigma_1}\sigma_2  -\bar{L}(\sigma_1,\sigma_2)] +\beta_2(r^2)g(\nabla_{h\sigma_1}\sigma_2,\mathcal{U})\mathcal{U} \\ 
    &-(\alpha_1 +\alpha_3)(r^2)\bar{C}(\sigma_1,\sigma_2) -\frac{(\beta_1+\beta_3)(r^2)}{2}[g(\sigma_1,\mathcal{U})\sigma_2 +g(\sigma_2,\mathcal{U})\sigma_1] \\
      &  -(\alpha_1 +\alpha_3)^\prime(r^2)g(\sigma_1,\sigma_2)\mathcal{U} -(\beta_1+\beta_3)^\prime(r^2)g(\sigma_1,\mathcal{U})g(\sigma_2,\mathcal{U})\mathcal{U}  \\ 
      & -\frac{\alpha_1(r^2)}{2} \mathcal{R}(\sigma_1,\sigma_2)\mathcal{U}.
  \end{split}
\end{equation*}

According to the decomposition into horizontal and vertical parts, we put
$$\bar{\nabla} _{X^{h}}Y^{h}= \nabla_{X^{h}}Y +h\{P_{hh}(X,Y)\} +v\{Q_{hh}(X,Y)\},$$
where $P_{hh}$ and $Q_{hh}$ are $(2,1)$-tensor sections on $\pi_0^*TM$. It follows that
\begin{eqnarray*}
  G(\bar{\nabla} _{X^{h}}Y^{h},Z^h) &=& (\alpha_1 +\alpha_3)(r^2)g(\nabla_{X^{h}}Y+P_{hh}(X,Y),Z) \\ &&+(\beta_1+\beta_3)(r^2)g(\nabla_{X^{h}}Y+P_{hh}(X,Y),\mathcal{U})g(Z,\mathcal{U}) \\
  &&+\alpha_2(r^2)g(Q_{hh}(X,Y),Z) +\beta_2(r^2)g(Q_{hh}(X,Y),\mathcal{U})g(Z,\mathcal{U}), \\
  G(\bar{\nabla} _{X^{h}}Y^{h},Z^v) &=& \alpha_2(r^2) g(\nabla_{X^{h}}Y+P_{hh}(X,Y),Z) +\beta_2(r^2) g(\nabla_{X^{h}}Y+P_{hh}(X,Y),\mathcal{U})g(Z,\mathcal{U}) \\
  &&+\alpha_1(r^2)g(Q_{hh}(X,Y),Z) +\beta_1(r^2)g(Q_{hh}(X,Y),\mathcal{U})g(Z,\mathcal{U}).
\end{eqnarray*}
Comparing the two preceding identities with the identities
$$G(\bar{\nabla} _{X^{h}}Y^{h},Z^{h})=g(P_1(X,Y),Z), \quad G(\bar{\nabla} _{X^{h}}Y^{h},Z^{v})=g(P_2(X,Y),Z),$$
we have the following 
\begin{equation*}
  \left\{
  \begin{array}{rl}
    P_1(X,Y)=&(\alpha_1 +\alpha_3)(r^2)(\nabla_{X^{h}}Y+P_{hh}(X,Y))\\
    & +(\beta_1+\beta_3)(r^2)g(\nabla_{X^{h}}Y+P_{hh}(X,Y),\mathcal{U})\mathcal{U} \\
    & +\alpha_2(r^2) Q_{hh}(X,Y) +\beta_2(r^2)g(Q_{hh}(X,Y),\mathcal{U})\mathcal{U}, \\
    &\\
    P_2(X,Y)=& \alpha_2(r^2) (\nabla_{X^{h}}Y+P_{hh}(X,Y)) \\
    &+\beta_2(r^2) g(\nabla_{X^{h}}Y+P_{hh}(X,Y),\mathcal{U})\mathcal{U} \\
    & +\alpha_1(r^2)Q_{hh}(X,Y) +\beta_1(r^2)g(Q_{hh}(X,Y),\mathcal{U})\mathcal{U}.
  \end{array}
  \right.
\end{equation*}
Substituting $P_1(X,Y)$ and $P_2(X,Y)$ into the preceding equations, we find
\begin{equation}\label{L-C-calc1}
  \left\{
  \begin{array}{rl}
   & (\alpha_1 +\alpha_3)(r^2)P_{hh}(X,Y) +(\beta_1+\beta_3)(r^2)g(P_{hh}(X,Y),\mathcal{U})\mathcal{U} \\
    & +\alpha_2(r^2) Q_{hh}(X,Y) +\beta_2(r^2)g(Q_{hh}(X,Y),\mathcal{U})\mathcal{U} \\
   = & -\alpha_2(r^2)[\mathcal{R}(X,\mathcal{U})Y-\bar{B}(X,Y)  \\ 
      & -\bar{C}(Y,\mathcal{R}(X,\mathcal{U})\mathcal{U})-\bar{C}(X,\mathcal{R}(Y,\mathcal{U})\mathcal{U})],\\
      & \\
    & \alpha_2(r^2) P_{hh}(X,Y) +\beta_2(r^2) g(P_{hh}(X,Y),\mathcal{U})\mathcal{U} \\
    & +\alpha_1(r^2)Q_{hh}(X,Y) +\beta_1(r^2)g(Q_{hh}(X,Y),\mathcal{U})\mathcal{U}\\
    =&-\alpha_2(r^2)\bar{L}(X,Y) -(\alpha_1 +\alpha_3)(r^2)\bar{C}(X,Y)  \\ 
      &  -(\alpha_1 +\alpha_3)^\prime(r^2)g(X,Y)\mathcal{U} -(\beta_1+\beta_3)^\prime(r^2)g(X,\mathcal{U})g(Y,\mathcal{U})\mathcal{U}  \\ 
      & -\frac{(\beta_1+\beta_3)(r^2)}{2}[g(X,\mathcal{U})Y +g(Y,\mathcal{U})X]  -\frac{\alpha_1(r^2)}{2} \mathcal{R}(X,Y)\mathcal{U}.
  \end{array}
  \right.
\end{equation}
Making the scalar product of the preceding two equations in \eqref{L-C-calc1} by $\mathcal{U}$, we find
\begin{equation*}
  \left\{
  \begin{array}{rl}
   & (\phi_1 +\phi_3)(r^2)g(P_{hh}(X,Y),\mathcal{U}) +\phi_2(r^2) g(Q_{hh}(X,Y),\mathcal{U}) \\
   = & -\alpha_2(r^2)g(\mathcal{R}(X,\mathcal{U})Y,\mathcal{U})  ,\\
      & \\
    & \phi_2(r^2)  g(P_{hh}(X,Y),\mathcal{U})+\phi_1(r^2)g(Q_{hh}(X,Y),\mathcal{U})\\
    =& -r^2(\alpha_1 +\alpha_3)^\prime(r^2)g(X,Y) -r^2(\beta_1+\beta_3)^\prime(r^2)g(X,\mathcal{U})g(Y,\mathcal{U})  \\ 
      & -(\beta_1+\beta_3)(r^2)g(X,\mathcal{U}) g(Y,\mathcal{U}),
  \end{array}
  \right.
\end{equation*}
which yields
\begin{equation*}
  \left\{
  \begin{array}{rl}
  g(P_{hh}(X,Y),\mathcal{U})= & \frac{1}{\phi}\{-\alpha_1(r^2)\phi_2(r^2)g(\mathcal{R}(X,\mathcal{U})Y,\mathcal{U}) +r^2(\alpha_1 +\alpha_3)^\prime(r^2)\phi_2(r^2)g(X,Y) \\
  & +[r^2(\beta_1+\beta_3)^\prime(r^2) +(\beta_1+\beta_3)(r^2)]\phi_2(r^2) g(X,\mathcal{U}) g(Y,\mathcal{U}),\\
      & \\
   g(Q_{hh}(X,Y),\mathcal{U})= & \frac{1}{\phi}\{\alpha_2(r^2)\phi_2(r^2)g(\mathcal{R}(X,\mathcal{U})Y,\mathcal{U}) -r^2(\alpha_1 +\alpha_3)^\prime(r^2)(\phi_1 +\phi_3)(r^2)g(X,Y) \\
  & -[r^2(\beta_1+\beta_3)^\prime(r^2) +(\beta_1+\beta_3)(r^2)](\phi_1 +\phi_3)(r^2) g(X,\mathcal{U}) g(Y,\mathcal{U}).
  \end{array}
  \right.
\end{equation*}

Substituting from the two last identities into \eqref{L-C-calc1} and solving the system with $P_{hh}(X,Y)$ and $Q_{hh}(X,Y)$ as indeterminate, we find the desired expressions of $P_{hh}(X,Y)$ and $Q_{hh}(X,Y)$ in the proposition. \cqfd

\medskip

\emph{Proof of Lemma \ref{Lem-Lie-hor}:}
We shall prove the second identity, the proof of the two other identities being similar. We have 
$$\mathcal{L}_{\xi ^{h}}G(X^{h},Y^{v})=G(\bar{\nabla }_{X^{h}}\xi ^{h},Y^{v})+G(\bar{\nabla }_{Y^{v}}\xi ^{h},X^{h}).$$ 
Using Proposition \ref{lev-civ-con}, we get
$$\begin{array}{rl}
G(\bar{\nabla }_{X^{h}}\xi ^{h},Y^{v})=&\alpha _{2}(r{{}^2})g(\nabla _{X^{h}}\xi ,Y)+\beta _{2}(r{{}^2})g(Y,\mathcal{U})g(\nabla _{X^{h}}\xi,\mathcal{U})\\
&-(\alpha _{1}+\alpha _{3})^{\prime }(r{{}^2})g(X,\xi )g(Y,\mathcal{U}) -(\alpha _{1}+\alpha _{3})(r{{}^2})C(X,Y,\xi )\\ 
&-(\beta _{1}+\beta _{3})^{\prime }(r{{}^2})g(X,\mathcal{U})g(Y,\mathcal{U})g(\xi,\mathcal{U}) \\ 
&-\frac{(\beta _{1}+\beta _{3})(r{{}^2})}{2}\left( g(\xi ,Y)g(X,\mathcal{U})+g(X,Y)g(\xi,\mathcal{U})\right) \\
&-\frac{\alpha _{1}(r{{}^2})}{2}\mathcal{R}(X,\xi ,\mathcal{U},Y)-\alpha _{2}(r{{}^2})L(X,Y,\xi )
\end{array}%
$$

and

$$%
\begin{array}{rl}
G(\bar{\nabla }_{Y^{v}}\xi ^{h},X^{h})=&(\alpha _{1}+\alpha_{3})^{\prime }(r{{}^2})g(\xi ,X)g(Y,\mathcal{U})+(\alpha _{1}+\alpha _{3})(r{{}^2}
)C(X,Y,\xi ) \\ 
&+\frac{(\beta _{1}+\beta _{3})}{2}(r{{}^2})\left(g(Y,\xi )g(X,\mathcal{U})+g(X,Y)g(\xi,\mathcal{U})\right) \\
& +(\beta _{1}+\beta_{3})^{\prime }(r{{}^2})g(X,\mathcal{U})g(Y,\mathcal{U})g(\xi,\mathcal{U}) \\ 
&+\frac{\alpha _{1}(r{{}^2})}{2}(\mathcal{R}(\mathcal{U},Y,\xi ,X)-B(Y,\xi ,X,\mathcal{U})\\
&-C(Y,\mathcal{R}(\mathcal{U},\xi )\mathcal{U},X)-C(\xi ,\mathcal{R}(\mathcal{U},Y)\mathcal{U},X))%
\end{array}%
$$

Summing up the two preceding formulas, we find the required identity. \cqfd

\medskip

\emph{Proof of Lemma \ref{Lem-Lie-ver}:}
We shall prove the first identity, the proof of the two other identities being similar. We have 
$$\mathcal{L}_{\xi ^{v}}G(X^{h},Y^{h})=G(\bar{\nabla }_{X^{h}}\xi ^{v},Y^{h})+G(\bar{\nabla }_{Y^{h}}\xi ^{v},X^{h}).$$ 
Using Proposition \ref{lev-civ-con}, we get

$$%
\begin{array}{rl}
G(\bar{\nabla }_{X^{h}}\xi ^{v},Y^{h})=&\alpha _{2}(r{{}^2})(g(\nabla _{X^h}\xi ,Y)+L(X,Y,\xi ))+\beta _{2}(r{{}^2})g(Y,\mathcal{U})g(\nabla _{X^h}\xi,\mathcal{U}) \\ 
&+\frac{\alpha _{1}(r{{}^2})}{2}(\mathcal{R}(\mathcal{U},\xi ,X,Y)-B(X,\xi ,Y,\mathcal{U})-C(X,Y,\mathcal{R}(\mathcal{U},\xi )\mathcal{U}) \\ 
&-C(\xi ,Y,\mathcal{R}(\mathcal{U},X)\mathcal{U})) +(\alpha _{1}+\alpha _{3})^{\prime }(r{{}^2})g(X,Y)g(\xi,\mathcal{U}) \\ 
&+(\alpha _{1}+\alpha _{3})(r{{}^2})C(X,Y,\xi ) +(\beta _{1}+\beta _{3})^{\prime }(r{{}^2})g(X,\mathcal{U})g(Y,\mathcal{U})g(\xi,\mathcal{U})\\
&+\frac{(\beta _{1}+\beta _{3})(r{{}^2})}{2}\left( g(X,\xi )g(Y,\mathcal{U})+g(Y,\xi )g(X,\mathcal{U})\right).%
\end{array}%
$$

Taking the symmetric version in $X$ and $Y$ of the preceding identity and summing up, we get the desired identity. \cqfd

\medskip

\emph{Proof of Lemma \ref{Lem-Lie-com}:}
We shall prove the first identity, the proof of the two other identities being similar. Since the complete lift is defined by $\xi ^{c}=\xi ^{h}+v\{\nabla_\zeta \xi \}$, then the Lie derivative of $G(X^{h},Y^{h})$ along $\xi ^{c}$ is given by
\begin{eqnarray*}
\mathcal{L}_{\xi ^{c}}G(X^{h},Y^{h}) &=&\mathcal{L}_{\xi ^{h}}G(X^{h},Y^{h}) +\mathcal{L}_{v\{\nabla_\zeta \xi \}}G(X^{h},Y^{h}) \\
&=&\mathcal{L}_{\xi ^{h}}G(X^{h},Y^{h})+G(\nabla _{X^{h}}(v\{\nabla_\zeta \xi \})^{v},Y^{h})+G(\nabla _{Y^{h}}(v\{\nabla_\zeta \xi \})^{v},X^{h}).
\end{eqnarray*}
But 
$$G(\nabla _{X^{h}}(\nabla _{\zeta }\xi )^{v},Y^{h})=\sum_iu^{i}G\left(\nabla _{X^{h}}\left(\nabla _{\left(\frac{\partial }{\partial x^{i}}\right)^h}\xi
\right)^{v},Y^{h}\right)-\sum_{i,j,k} u^{j}X^{k}\Gamma _{jk}^{i}G\left(\left(\nabla _{\frac{\partial}{\partial x^{i}}}\xi \right)^{v},Y^{h}\right).$$
We have, on one hand, using Proposition \ref{lev-civ-con}
\arraycolsep0pt
$$%
\begin{array}{rl}
\sum_iu^{i}&G\left(\nabla _{X^{h}}\left(\nabla _{\left(\frac{\partial }{\partial x^{i}}\right)^h}\xi
\right)^{v},Y^{h}\right)=\\
=&\alpha _{2}(r{{}^2})[g(\nabla _{X^h}(\nabla _{\zeta }\xi ),Y)+g(\nabla _{Y^h}(\nabla _{\zeta }\xi),X)+2L(X,Y,\nabla _{\zeta }\xi )] \\ 
&+\beta _{2}(r{{}^2})[g(X,\mathcal{U})g(\nabla _{Y^h}(\nabla _{\zeta}\xi ),\mathcal{U})+g(Y,\mathcal{U})g(\nabla _{X^h}(\nabla_{\zeta}\xi ),\mathcal{U})]\\
&+2(\beta _{1}^{\prime }+\beta_{3}^{\prime })(r{{}^2})g(X,\mathcal{U})g(Y,\mathcal{U})g(\nabla _{\zeta}\xi ,\mathcal{U}) \\ 
&+\alpha _{1}(r{{}^2})[2B(X,Y,\mathcal{U},\nabla _{\zeta }\xi )-B(X,\nabla _{\zeta }\xi ,Y,\mathcal{U})-B(Y,\nabla _{\zeta }\xi ,X,\mathcal{U})] \\ 
&+2(\alpha _{1}+\alpha _{3})^{\prime }(r{{}^2})g(X,Y)g(\nabla _{\zeta }\xi,\mathcal{U} )+2(\alpha _{1}+\alpha _{3})(r{{}^2})C(X,Y,\nabla _{\zeta }\xi ) \\ 
&+(\beta _{1}+\beta _{3})(r{{}^2})[g(X,\nabla _{\zeta }\xi )g(Y,\mathcal{U})+g(Y,\nabla _{\zeta }\xi )g(X,\mathcal{U})],
\end{array}%
$$
and, on the other hand,
$$\begin{array}{rl}
-\sum_{i,j,k}u^{j}&X^{k}\Gamma _{jk}^{i}G((\nabla _{\left(\frac{\partial }{\partial x^{i}}\right)^h}\xi )^{v},Y^{h}) =\\
=&-\sum_{i,j,k}u^{j}X^{k}\Gamma _{jk}^{i}[\alpha_{2}(r{{}^2})g(\nabla _{\left(\frac{\partial }{\partial x^{i}}\right)^h}\xi ,Y)+\beta _{2}(r{{}^2})g(Y,\mathcal{U})g(\nabla _{\left(\frac{\partial }{\partial x^{i}}\right)^h}\xi ,\mathcal{U})] \\
=&-\alpha _{2}(r{{}^2})g(\nabla _{h\{\nabla _{X^h}\mathcal{U}\}}\xi ,Y)-\beta _{2}(r{{}^2})g(Y,\mathcal{U})g(\nabla _{h\{\nabla _{X^h}\mathcal{U}\}}\xi,\mathcal{U} )]
\end{array}$$

Then, summing up the two last identities and using the definition of the second covariant derivative and using again Proposition \ref{lev-civ-con} for the expression $\mathcal{L}_{\xi ^{h}}G(X^{h},Y^{h})$, we obtain the required expression of $\mathcal{L}_{\xi ^{c}}G(X^{h},Y^{h})$.
\cqfd

\medskip

\emph{Proof of Lemma \ref{Lem-Lie-iota}:}
We will prove the first identity, the proof of the two others being similar. We have, on one hand,
$$\begin{array}{rl}
    \mathcal{L}_{\iota P}G(X^{h},Y^{h})= & u^{i}\mathcal{L}_{(P(\partial _{i}))^{v}}G(X^{h},Y^{h})+X^{h}(u^{i})G((P(\partial
_{i}))^{v},Y^{h}) \\
     & +Y^{h}(u^{i})G((P(\partial _{i}))^{v},X^{h})
  \end{array}
$$
and on the other hand, using Proposition \ref{lev-civ-con}, we get
$$\begin{array}{rl}
X^{h}(u^{i})G((P(\partial _{i}))^{v},Y^{h}) =&-u^{j}X^{k}\Gamma _{jk}^{i}[\alpha _{2}(r{{}^2})g(P(\partial _{i}),Y)+\beta _{2}(r{{}^2}
)g(P(\partial _{i}),\mathcal{U})g(Y,\mathcal{U})] \\
=&-\alpha _{2}(r{{}^2})g(P(\nabla _{X^h}\mathcal{U}),Y)-\beta _{2}(r{{}^2})g(P(\nabla _{X^h}\mathcal{U}),\mathcal{U})g(Y,\mathcal{U}).
\end{array}
$$
Then
$$%
\begin{array}{rl}
\mathcal{L}_{\iota P}G(X^{h},Y^{h})=&\alpha _{2}(r{{}^2})[g(\nabla _{X}P(\mathcal{U}),Y)-g(P(\nabla _{X^h}\mathcal{U}),Y) \\ 
&+g(\nabla _{Y^h}P(\mathcal{U}),X)-g(P(\nabla _{Y^h}\mathcal{U}),X)+2L(X,Y,P(\mathcal{U}))] \\ 
&+\beta _{2}(r{{}^2})(g(X,\mathcal{U})g(\nabla _{Y^h}P(\mathcal{U}),\mathcal{U})-g(P(\nabla_{Y^h}\mathcal{U}),\mathcal{U})g(X,\mathcal{U}) \\ &+g(Y,\mathcal{U})g(\nabla _{X^h}P(\mathcal{U}),\mathcal{U})-g(P(\nabla _{X^h}\mathcal{U}%
),\mathcal{U})g(Y,,\mathcal{U}))\\
&+2(\beta _{1}+\beta _{3})^{\prime }(r{{}^2})g(X,\mathcal{U})g(Y,,\mathcal{U})g(P(\mathcal{U}),\mathcal{U}) \\ 
&+\alpha _{1}(r{{}^2})(2B(X,Y,\mathcal{U},P(\mathcal{U}))-B(X,P(\mathcal{U}),Y,\mathcal{U})-B(Y,P(\mathcal{U}),X,\mathcal{U})) \\ 
&+2(\alpha _{1}^{\prime }+\alpha _{3}^{\prime })(r{{}^2})g(X,Y)g(P(\mathcal{U}),\mathcal{U})+2(\alpha _{1}+\alpha _{3})(r{{}^2})C(X,Y,P(\mathcal{U})) \\ 
&+(\beta _{1}+\beta _{3})(r{{}^2})\left( g(X,P(\mathcal{U}))g(Y,\mathcal{U})+g(Y,P(\mathcal{U}))g(X,\mathcal{U})\right).
\end{array}%
$$
\cqfd

\medskip

\emph{Proof of Lemma \ref{Lem-Lie-tau}:}
We prove the first expression of the above system, the proof of the other two expressions being similar. We have, on one hand,
$$%
\begin{array}{rl}
\mathcal{L}_{\tau P}G(X^{h},Y^{h})= & G(\nabla _{X^{h}}u^{i}(P(\frac{\partial }{\partial x^{i}}))^{h},Y^{h})+G(\nabla _{Y^{h}}u^{i}(P(\frac{%
\partial }{\partial x^{i}}))^{h},X^{h}) \\ 
=& u^{i}\mathcal{L}_{(P(\frac{\partial }{\partial x^{i}}))^{h}}G(X^{h},Y^{h})-X^{j}u^{k}\Gamma _{jk}^{i}G((P(\frac{\partial }{\partial x^{i}}))^{h},Y^{h}) \\ 
& -Y^{j}u^{k}\Gamma _{jk}^{i}G((P(\frac{\partial }{\partial x^{i}}))^{h},X^{h}) \\ 
=& u^{i}\mathcal{L}_{(P(\frac{\partial }{\partial x^{i}}))^{h}}G(X^{h},Y^{h}) \\ 
&-(\alpha _{1}+\alpha _{3})(r^{2})[g(P(\nabla _{X^{h}}\mathcal{U}),Y)+g(P(\nabla _{Y^{h}}\mathcal{U}),X)] \\ 
& -(\beta _{1}+\beta _{3})(r^{2})[g(P(\nabla _{X^{h}}\mathcal{U}),\mathcal{U})g(Y,\mathcal{U})+g(P(\nabla _{Y^{h}}\mathcal{U}),\mathcal{U})g(X,\mathcal{U}%
)] 
\end{array}%
$$
and, on the other hand, we have by Proposition \ref{lev-civ-con}
$$%
\begin{array}{rl}
u^{i}\mathcal{L}_{(P(\frac{\partial }{\partial x^{i}}))^{h}}G(X^{h},Y^{h})= & (\alpha _{1}+\alpha _{3})(r{{}^{2}})[g(\nabla _{X^{h}}P(\mathcal{U}%
),Y)+g(\nabla _{Y^{h}}P(\mathcal{U}),X)] \\ 
& +(\beta _{1}+\beta _{3})(r{{}^{2}})[g(\nabla _{X^{h}}P(\mathcal{U}),\mathcal{U})g(Y,\mathcal{U}) \\ 
&+g(\nabla _{Y^{h}}P(\mathcal{U}),\mathcal{U})g(X,\mathcal{U})]\\ 
& -\alpha _{2}\left( r{{}^{2}}\right) [\mathcal{R}(X,\mathcal{U},P(\mathcal{U}),Y)+\mathcal{R}(Y,\mathcal{U},P(\mathcal{U}),X) \\ 
&+2C(P(\mathcal{U}),\mathcal{R}(\mathcal{U},X)\mathcal{U},Y)\\ 
& +2C(P(\mathcal{U}),\mathcal{R}(\mathcal{U},Y)\mathcal{U},X)+2C(X,\mathcal{R}(\mathcal{U},P(\mathcal{U}))\mathcal{U},Y)].
\end{array}%
$$
Substituting, we get the desired expression.
\cqfd

\end{appendix}

 \begin{flushleft}
\textsc{Department of Mathematics, \\
Faculty of sciences Dhar El Mahraz,\\
Sidi Mohamed Ben Abdallah University,\\
B.P. 1796, F\`es-Atlas,\\
Fez, Morocco.}\\
\textit{E-mail address}: mtk{\_}abbassi@Yahoo.fr; abderrahim.mekrami@gmail.com.
\end{flushleft}

\end{document}